\documentclass[12pt]{amsart}
\usepackage {amscd}
\usepackage[german,english]{babel}
\usepackage {amsfonts}
\usepackage[all]{xy}
\usepackage{amssymb,amsmath, amsthm,latexsym,verbatim}
\usepackage{a4wide}
\usepackage[hypertex]{hyperref}\usepackage{enumitem}

\newcommand{\C}{\mathbb{C}}
\newcommand{\R}{\mathbb{R}}

\newcommand{\N}{\mathbb{N}}

\newcommand{\s}{\operatorname{s}}

\newcommand{\kL}{\mathfrak{k}}
\newcommand{\mL}{\mathfrak{m}}
\newcommand{\aL}{\mathfrak{a}}
\newcommand{\gL}{\mathfrak{g}}
\newcommand{\pL}{\mathfrak{p}}
\newcommand{\nL}{\mathfrak{n}}
\newcommand{\hL}{\mathfrak{h}}

\newcommand{\Sym}{\operatorname{Sym}}
\newcommand{\Gl}{\operatorname{GL}}

\newcommand{\an}{\operatorname{an}}
\newcommand{\Id}{\operatorname{Id}}

\newcommand{\Spin}{\operatorname{Spin}}
\newcommand{\reg}{\operatorname{reg}}

\newcommand{\Res}{\operatorname{Res}}
\newcommand{\SU}{\operatorname{SU}}
\newcommand{\Ad}{\operatorname{Ad}}
\newcommand{\PSl}{\operatorname{PSL}}
\newcommand{\Sl}{\operatorname{SL}}
\newcommand{\PSO}{\operatorname{PSO}}
\newcommand{\SO}{\operatorname{SO}}

\newcommand{\CC}{\operatorname{C}}
\newcommand{\gr}{\operatorname{gr}}

\newcommand{\vol}{\operatorname{vol}}
\newcommand{\Tr}{\operatorname{Tr}}

\newcommand{\End}{\operatorname{End}}
\newcommand{\Real}{\operatorname{Re}}
\newcommand{\sym}{\operatorname{sym}}
\newtheorem {thrm}{Theorem}[section]
\newtheorem {prop}[thrm] {Proposition}
\newtheorem {lem}[thrm] {Lemma}

\theoremstyle{definition}

\theoremstyle{remark}
\newtheorem {bmrk}[thrm] {Remark}

\begin{document}
\author{Jonathan Pfaff}
\address{Universit\"at Bonn\\
Mathematisches Institut\\
Endenicher Alle 60\\
D -- 53115 Bonn, Germany}
\email{pfaff@math.uni-bonn.de}
\title[]{Analytic torsion versus Reidemeister torsion on hyperbolic
3-manifolds with cusps}

\begin{abstract}
For a non-compact hyperbolic 3-manifold with cusps we prove an explicit
formula that relates the regularized analytic torsion associated to the even
symmetric powers of the standard
representation of $\Sl_2(\C)$ to 
the corresponding Reidemeister torsion. Our proof rests on an expression of the
analytic 
torsion in terms of special values of Ruelle zeta functions as well as on 
recent work of Pere Menal-Ferrer and Joan Porti. 
\end{abstract}

\maketitle
\setcounter{tocdepth}{1}

\section{Introduction}
\setcounter{equation}{0}
Let $X$ be a hyperbolic manifold with cusps of odd dimension $d$. Then 
$X$ is not compact but has finite volume. In a
previous 
publication \cite{MP2} we have introduced the analytic torsion $T_X(\rho)$ with 
coefficients in the flat vector bundle $E_\rho$ which is obtained by
restricting 
a finite-dimensional complex representation $\rho$ of $G:=\Spin(d,1)$ to the
fundamental group 
$\Gamma\subset G$ of $X$. The aim of this paper is to relate the torsion
$T_X(\rho)$ to 
the corresponding Reidemeister torsion invariants for the case that $X$ is
3-dimensional.

In order to motivate our results, let us first
recall the situation on a closed odd-dimensional Riemannian manifold
$\left(M,g\right)$. Let
$\Gamma$ denote the fundamental group of $M$ and let $\widetilde{M}$ be the
universal
covering space of $M$. Let $\rho$ be a finite-dimensional representation of
$\Gamma$ on a complex vector space $V_\rho$. Moreover assume that $\rho$ is
unimodular, which
means that $\rho$ satisfies
$\left|\det\rho(\gamma)\right|=1$ for all $\gamma\in\Gamma$. Let
$E_\rho:=\widetilde{M}\times_\rho V_\rho$ be the associated flat vector bundle
over $M$. Pick a Hermitian fibre metric $h$ in $E_\rho$. Then the 
analytic torsion $T_M(\rho)\in\R^+$ is a spectral invariant of $E_\rho$ 
which depends on the metrics on $M$ and $E_\rho$. It is defined 
as a weighted product over the zeta-determinants of the Hodge-Laplace 
operators which act on the $E_\rho$-valued p-forms on $M$, see \cite[section
2]{Muellerzwei}. 
There exists a combinatorial counterpart of the analytic torsion, the so called
Reidemeister  torsion. The latter is constructed in a combinatorial way out of a
smooth triangulation of $M$. It depends on a choice of bases in the homology
groups 
$H_*(M,E_\rho)$ of $M$ with coefficients in the local system defined by $\rho$. 
However, via the Hodge-DeRham isomorphism and Poincar\'e duality, the metrics
$g$ and $h$ canonically define such bases. In this way one obtains a
combinatorial
invariant $\tau_M(\rho;h)$, the Reidemeister torsion associated  
to $(M,g)$ and the Hermitian vector bundle $(E_\rho,h)$, see \cite[section
1, section 2]{Muellerzwei}. Now the analytic torsion
and 
the Reidemeister torsion are equal, i.e. one has 
$T_M(\rho;h)=\tau_M(\rho;h)$. For the case that $\rho$ is unitary, this was
proved independently by Cheeger
\cite{Cheger} and M\"uller \cite{Muellereins}. The extension to unimodular
representations is due to M\"uller \cite{Muellerzwei}.

Let us now turn to the actual setup of this paper. We let $X$ be a hyperbolic
3-manifold 
which is not compact but of finite volume. If $G:=\Sl_2(\C)$,
$K:=\SU(2)$, then 
$\widetilde X:=G/K$ can be identified with the hyperbolic 3-space and there
exists a discrete, torsion free 
subgroup $\Gamma$ of $G$ such that $X=\Gamma\backslash\widetilde X$. One can
identify
$\Gamma$ with the fundamental group of $X$. Throughout this paper, we assume
that $\Gamma$ satisfies a certain condition, which is formulated in
equation \eqref{asGamma} below. 
Let $\rho$ be an irreducible finite-dimensional complex
representation of $G$. Restrict $\rho$ to $\Gamma$ and let $E_\rho$
be the associated flat vector-bundle over $X$. One can equip 
$E_\rho$ with a canonical metric, called admissible metric.
The associated Laplace operator $\Delta_p(\rho)$ on $E_\rho$-valued $p$-forms
has a
continuous spectrum and therefore, the heat operator $\exp(-t\Delta_p(\rho))$ 
is not trace class. So the usual zeta function regularization can not be used
to define the analytic torsion. However, picking up the concept of the $b$-trace
of Melrose, employed by Park in a similar context, in \cite{MP2} we introduced
the regularized
trace $\Tr_{\reg}\left(e^{-t\Delta_p(\rho)}\right)$ of the operators
$e^{-t\Delta_p(\rho)}$ and in this way we extended the definition of
the analytic torsion to 
the non-compact manifold $X$. These definitions will be reviewed in section
\ref{sectrdet}
below. Let $T_X(\rho)$ denote the analytic torsion on $X$ associated 
to $\rho$.

The aim of the present article is to find a suitable generalization of the
aforementioned Cheeger-M\"uller 
theorems to the specific non-compact situation of the hyperbolic 3-manifold $X$
with cusps. For $m\in\frac{1}{2}\mathbb{N}$ we let $\rho(m)$ denote the $2m$-th
symmetric
power 
of the standard representation of $\Sl_2(\C)$. Let 
$\overline{X}$ be the Borel-Serre compactifcation of $X$. We recall that 
$\overline{X}$ is a compact smooth manifold with boundary and that $X$ is
diffeomorphic 
to the interior of $\overline{X}$. Moreover, $\overline{X}$ and $X$ are
homotopy-equivalent.
Thus, every representation $\rho:=\rho(m)$ of $G$ also defines a flat vector
bundle
$\overline{E_\rho}$ 
over $\overline{X}$. Now by our assumption
\eqref{asGamma} on $\Gamma$ and \cite[Proposition 2.8]{MePo1}, the 
cohomology $H^*(X,\rho)$ never
vanishes. Thus in order to define the Reidemeister torsion of
$\overline{E_\rho}$, one 
needs to fix bases in the homology $H_*(X,\rho)$. However, by \cite[Lemma
7.3]{MP2} the bundle
$E_\rho$ 
is $L^2$-acyclic and thus the metrics on $X$ and $E_\rho$ do not give 
such bases. This fact is a significant difference to the situation 
on a closed manifold described above and causes additional difficulties. To
overcome this problem, we use the normalized Reidemeister torsion which was
introduced by Menal-Ferrer and Porti \cite{MePo}. 
Recall that
the boundary of $\overline{X}$ is a disjoint union of finitely many tori $T_i$. 
For each $i$ fix a non-trivial cycle $\theta_i\in H_1(T_i;\mathbb{Z})$. 
By our assumption on the group $\Gamma$, it follows from \cite[Proposition
2.2]{MePo}
that the $\{\theta_i\}$ can be used do define a base in the homology
$H_*(X;\rho(m))$ for each 
$m\in\frac{1}{2}\mathbb{N}$. Denote the corresponding Reidemeister torsion 
by $\tau_X(\rho(m);\{\theta_i\})$. Then by \cite[Proposition 2.2]{MePo}
for each $m\in\mathbb{N}$ the quotient 
\begin{align}\label{NormTor}
\mathcal{T}_X(\rho(m)):=\frac{\left|\tau_X(\rho(m);\{\theta_i\})\right|}{
\left|\tau_X(\rho(2);\{
\theta_i\})\right|}
\end{align}
is independent of the choice of the $\{\theta_i\}$. As explained in
\cite[section 1, section 2]{MePo}, it is also independent of a given
spin-structure on $X$. The number $\mathcal{T}_X(\rho(m))$ is called normalized
Reidemeister torsion of $X$ associated to $\rho(m)$. We remark that our
parametrization of the 
representations $\rho(m)$ differs from the one used by Menal-Ferrer and Porti in
\cite{MePo} but is 
consistent with the notation of \cite{MP1}, \cite{MP2} and \cite{Pf}.
Menal-Ferrer and Porti
expressed the normalized 
Reidemeister torsion in terms of special values of Ruelle zeta functions
\cite[Theorem 5.8]{MePo}. This
relation
generalizes a result obtained by M\"uller for closed hyperbolic 3-manifolds,
\cite[equation 8.7, equation 8.8] {Muellerzwei} and 
is proved via a Dehn-approximation of $X$ by closed hyperbolic 3-manfiolds. 
Now, in analogy to \eqref{NormTor}, for $m\in\mathbb{N}$ we
define 
the normalized analytic torsion $\mathcal{T}^{\an}_X(\rho(m))$  by
\begin{align*}
\mathcal{T}_X^{\an}(\rho(m)):=\frac{T_X(\rho(m))}{T_X(\rho(2))}.
\end{align*}
Then our main result
can be 
stated in the following theorem.  
\begin{thrm}\label{Theorem1}
For $m\in\mathbb{N}$, $m\geq 2$ we define
\begin{align*}
c(m):=
&\frac{\prod_{j=1}^{m-1}\sqrt{(m+1)^2+m^2-j^2}+m}{\prod_{j=1}^{
m
}\sqrt{(m+1)^2+m^2-j^2}+m+1}\left(\frac{\sqrt{(m+1)^2+m^2}+m}{\sqrt{
(m+1)^2+m^2}+m+1}\right)^{\frac{1}{2}}.
\end{align*}
Let $\kappa(X)$ be the number of cusps of $X$. Then for $m\in\mathbb{N}$, $m\geq
2$  one has
\begin{align*}
\mathcal{T}^{\an}_X(\rho(m))=\left(\frac{c(m)}{c(2)}\right)^{\kappa(X)}
\mathcal{T}_X(\rho(m)).
\end{align*}
\end{thrm}
We first remark that neither the normalized Reidemeister torsion nor the 
quotient of the analytic torsions occuring in Theorem \ref{Theorem1} 
are trivial. In fact, each of them is exponentially growing as $m\to\infty$.
This 
follows for example from Theorem \ref{Theorem2} below or from the more 
general results of \cite{MP2}. The constants $c(m)$ are
a defect caused by the non-compactness of the manifold. They appear 
via the contribution of a certain non invariant distribution to the geometric
side
of the Selberg trace formula. 

In a furthcoming publication, Theorem \ref{Theorem1} will be applied to study
for fixed $m\in\mathbb{N}$ the asymptotic behaviour of the
torsion-growth 
in the cohomology $H^*(\Gamma_i,M_{\rho(m)})$ for towers of arithmetic
groups $\Gamma_i\subset \Sl_2(\C)$. Here $M_{\rho(m)}$ denotes a
lattice in the 
representation space $V_{\rho(m)}$ of $\rho(m)$ 
which is stable under the  $\Gamma_i$. In this way we will obtain a modified
extension of some results
of Bergeron and Venkatesh \cite{BV} 
to the noncompact case. More precisely, as it was already observed by Bergeron
and Venkatesh \cite[section 2]{BV} (see also \cite[equation 1.4]{Cheger}), the
size of the torsion subgroups is closely related to the
Reidemeister torsion . On the 
other hand, the analytic torsion $T_X(\rho(m))$ is accessible to compuations
since it can be computed in a rather explicit form via the Selberg trace
formula. 

As already indicated, our proof of Theorem \ref{Theorem1} is based on an 
expression of the analytic torsion in terms of special values of Ruelle zeta
functions. Our method to establish such a relation works for every
representation
$\rho(m)$, $m\in\frac{1}{2}\mathbb{Z}$.
Thus for $k\in\frac{1}{2}\mathbb{N}$ we let $\sigma_k$ be the representation of
$M:=\SO_2(\R)$ with 
highest weight $ke_2$ as in section \ref{sechyp}. Then we define the 
Ruelle zeta function $R(s,\sigma_k)$ as in equation \eqref{DefRuelle1}. The
infinite
product in \eqref{DefRuelle1} converges for 
$s\in\C$ with $\Real(s)>2$. We will prove the following theorem.
\begin{thrm}\label{Theorem2}
Let $m\in\mathbb{N}$. Then for $m\geq 3$ one has
\begin{align*}
\frac{T_X(\rho(m))}{T_X(\rho(2))}=\left(\frac{c(m)}{c(2)}\right)^{
\kappa(X)}\exp{
\left(
-\frac{1}{\pi}\vol{(X)}(m(m+1)-6)\right)}\prod_
{
k=3}^m\left|R(k,\sigma_k)\right|,
\end{align*}
where the constants $c(\rho(m))$ and $c(\rho(2))$ are as in Theorem
\ref{Theorem1}.
Similarly, for each $m\geq 1$ there exist constants $c(m+1/2)$, defined 
in \eqref{Defc} such that for $m\geq 2$ one has 
\begin{align*}
\frac{T_X(\rho(m+\frac{1}{2}))}{T_X(\rho(\frac{3}{2}))}=\left(\frac{c(m+\frac{1}
{2})}{
c(\frac{3}{2})}\right)^{
\kappa(X)}\exp{
\left(
-\frac{1}{\pi}\vol{(X)}(m(m+2)-3)\right)}\prod_
{
k=2}^m\left|R(k+\frac{1}{2},\sigma_{k+\frac{1}{2}})\right|.
\end{align*}
\end{thrm}
If one combines the first statement of the previous Theorem with the
corresponding result 
of Menal-Ferrer and Porti \cite[Theorem 5.8]{MePo}, 
Theorem \ref{Theorem1} follows immediately.

We note that one can not combine Theorem \ref{Theorem2} and the
corresponding result 
\cite[Theorem 5.8]{MePo} of Menal-Ferrer and Porti to deduce an analog 
of Theorem \ref{Theorem1} for the representations $\rho(m+1/2)$,
$m\in\mathbb{N}$, $m\geq 2$. The problem is that in this case the normalized
Reidemeister torsion and the Ruelle zeta functions occuring in
\cite[Theorem 5.8]{MePo} are defined 
with respect to an acylic spin-structure of $X$. In the setting of the present
article, 
this means  that one replaces the group $\Gamma\subset G$ by a suitable group
$\Gamma'\subset G$ 
which has the same image in $\PSl_2(\C)$ as $\Gamma$. Clearly, for 
each $k\in\mathbb{N}$ the Ruelle zeta functions 
$R(s,\sigma_k)$ 
remain the same under this change of the group $\Gamma$ since
each representation 
$\sigma_{k}$, $k\in\mathbb{N}$ descends to 
a representation of $\PSO_2(\R)$. However, this is no longer 
the case for the representations $\sigma_{k}$,
$k\in\mathbb{Z}-\frac{1}{2}\mathbb{Z}$. Furthermore, Theorem \ref{Theorem2}
can not be applied 
to a group $\Gamma'$ corresponding to an acyclic spin-structure of $X$ since
by \cite[Lemma 2.4]{MePo1} such a group never satisfies 
the assumption \eqref{asGamma}. This assmuption is yet needed 
for our compuations involving the Selberg trace formula and in order to apply
the results 
about the meromorphic continuation of the zeta functions
obtained in \cite{Pf}. 

We shall now explain our method to prove Theorem \ref{Theorem2}. We first recall
that on closed odd-dimensional hyperbolic manifolds Fried \cite{Fried} related
the behaviour 
of the Ruelle zeta function $R_\rho$ associated to a unitary representation
$\rho$
of $\Gamma$ to the corresponding analytic torsion $T_X(\rho)$. In particular, 
if $\rho$ is acyclic, which means that the cohomology $H^*(X,E_\rho)$ vanishes,
he showed that the function $R_\rho$ is regular 
at $0$ and that $R_\rho(0)=T_X(\rho)^2$. However, there 
is no obvious method of constructing non-trivial unitary representations 
of (cocompact) hyperbolic lattices $\Gamma$.  On the other hand,  if
$X$ is a closed $2n+1$-dimensional hyperbolic manifold and if $\rho$ is a
representation of $\Spin(2n+1,1)$ 
which is not invariant under the standard Cartan-involution $\theta$, then the
restriction of $\rho$ to 
the fundamental group $\Gamma\subset \Spin(2n+1,1)$ of $X$ is an acyclic
unimodular 
representation. Generalizing Fried's results, Br\"ocker
\cite{Brocker} and Wotzke \cite{Wotzke} have shown that for such
representations $\rho$ 
the Ruelle zeta function 
$R_\rho$ is regular at $0$ and that one has $R_\rho(0)=T_X(\rho)^2$. 

We establish a generalization of the results of Br\"ocker and Wotzke to the
non-compact 
hyperbolic 3-manifold $X$ which is sufficient to prove Theorem \ref{Theorem2}
and thereby Theorem \ref{Theorem1}. The first main problem ist that the Ruelle
zeta
function $R_{\rho(m)}$ is 
a priori defined only for $s\in\C$ with $\Real(s)>2$. However, as a special case
of our results obtained in \cite{Pf}, it follows that $R_{\rho(m)}$ admits a
meromorphic continuation to 
the entire complex plane. This is the first step in the proof. The main
technical issue of this paper is now to 
relate the behaviour of $R_{\rho(m)}$ at zero to the regularized analytic
torsion
$T_X(\rho(m))$. Let $\rho(m)_\theta:=\rho(m)\circ\theta$. We prove the following
proposition. 
\begin{prop}\label{Prop1}
For $m\in\mathbb{N}$ let the constant $c(m)$ be as in Theorem \ref{Theorem1}.
Then
\begin{align*}
&{T_X(\rho(m))}^4\\
=&c(m)^{4\kappa(X)}\frac{\mathbf{C}(m:0)}{\mathbf{C}(m+1:0)}\lim_
{s\to
0}\bigr(R_{\rho(m)}(s)R_{\rho(m)_{\theta}}(s)\frac{\mathbf{C}(m+1:m-s)}{\mathbf{
C}(m:m+1-s)}\Gamma^{
-2\kappa(X)}(s-1)\bigl).
\end{align*}
Here the functions $\mathbf{C}(k:s)$ are meromorphic functions of $s$ which are
constructed
out of the scattering determinant associated to the representation $\sigma$ of
$M$ 
with highest weight $ke_2$ and a certain $K$-type. They are defined in section
\ref{secfe}.
For $m\in\frac{1}{2}\mathbb{N}$, a similar formula holds.
\end{prop}
Due to the presence of the scattering term and the $\Gamma$-factor, Proposition
\ref{Prop1}
does not imply that the Ruelle zeta function $R_{\rho(m)}$ 
is regular at $0$. However, from Proposition \ref{Prop1} one can deduce Theorem
\ref{Theorem2}
which is much more explicit. 

We remark that for odd-dimensional hyperbolic
manifolds 
with cusps and for unitary representatiosn of $\Gamma$, Park studied the
relation 
between the behaviour of the Ruelle zeta function at $0$ and the analytic
torsion \cite{Park}. However, his results can not be applied here since the
representations 
$\rho(m)$ are not unitary. Moreover the paper \cite{Park} decisively uses the
results of 
the earlier paper of Gon and Park \cite{Gon} on Selberg and Ruelle zeta
functions and the results of this paper do not imply that the Ruelle zeta
function 
$R_{\rho(m)}$ admits a meromorphic continuation to $\C$. Furthermore, in the
3-dimensional case, the paper \cite{Gon}
only covers the Selberg and Ruelle zeta functions associated to the fundamental
representations 
$\sigma_0$, $\sigma_1$ of $M$  and it is unclear whether the methods of
Gon and Park can be applied to other representations of $M$ since among other
things they use a special type of a Paley-Wiener theorem which presently exists
only for 
the fundamental representations of $K$. 
The proof of our main results is yet
based on the
meromorphic continuation of the Ruelle and Selberg zeta functions associated to
any representation $\sigma_k$, $k\in\mathbb{Z}$ as well as on their relation to
geometric differential operators on $X$. These results have been established in 
our preceding paper \cite{Pf} in the more general context of odd-dimensional
hyperbolic manifolds with cusps. We want to point out that,
as well as in the preceding 
paper \cite{Pf}, a lot of the methods used in the present article had 
been developed by Bunke and Olbrich \cite{Bunke} for the closed case and are
generalized here to
the non-compact situation. 
This generalization is made possible by the work of Hoffmann who proved 
an invariant trace formula \cite{Hoffmann2} and who determined the Fourier
transform of the
associated 
weighted orbital integral \cite{Hoffmann}.

To prove Proposition \ref{Prop1}, we first express the analytic torsion
$T_X(\rho(m))$
as a weighted product of graded determinants associated to differential
operators 
$A(\sigma)$ for certain $\sigma\in\hat{M}$.  Here the $A(\sigma)$ are of Laplace
type 
and act on graded locally homogeneous vector bundles  $E(\sigma)$ over $X$. 
By the same argument as in the closed case \cite{Wotzke}, \cite{Muellerzwei} 
the product $R_{\rho(m)}(s) R_{\rho(m)_\theta}(s)$ can be expressed
as a
weighted 
product of Selberg zeta functions $S(s,\sigma)$ with shifted arguments for the
same set 
of representations $\sigma\in\hat{M}$.  To relate the analytic torsion to the
Ruelle zeta function, 
we first prove a determinant formula which expresses the Selberg zeta function
$S(s,\sigma)$ 
by the graded determinant of $A(\sigma)+s^2$. The prove is based on an explicit
evaluation of 
the Laplace-Mellin transform of each term occuring on the geometric side of the
Selberg 
trace formula applied to a particular test function $h_t^\sigma$. However, in
contrast to the closed case, the determinant 
formula can only be applied to $s\in\C$ with $\Real(s)$ and $\Real(s^2)$
sufficiently large. Thus to complete the proof of Proposition \ref{Prop1}, we
also need to establish a functional equation for the symmetric Selberg zeta
functions. Via the
functional equations the 
scattering terms appear in Proposition \ref{Prop1}.

This paper is organized as follows. In section \ref{sechyp} we fix notations
and 
recall some basic facts about hyperbolic 3-manifolds. In section
\ref{seczeta} we briefly recall the definition of the Ruelle and Selberg
zeta functions. The definition of the 
regularized traces and the analytic torsion are reviewed in section
\ref{sectrdet}. In 
sections \ref{secdetfrml} and \ref{secfe} we establish the determinant
formula respectively the 
functional equations of the symmetric Selberg zeta functions. The proof of our
main results 
is completed in the final section \ref{secpr}. 

\bigskip
{\bf Acknowledgement.}
This paper contains parts of the author's PhD thesis. He would like to thank
his 
supervisor Prof. Werner M\"uller for his constant support and for helpful
suggestions.

\section{Hyperbolic 3-manifolds with cusps}\label{sechyp}
Let $\mathbb{H}^3$ denote the hyperbolic $3$-space equipped with the hyperbolic 
metric of constant curvature $-1$. 
Let $G=\Sl_{2}(\mathbb{C})$, regarded as a real Lie group, and let $K=\SU(2)$.
Then $K$ is a maximal compact subgroup of $G$. The groups $G$ and $K$ can be
identified 
with the groups $\Spin(3,1)$ and $\Spin(3)$ and there is a canonical 
isomorphism $\mathbb{H}^3\cong G/K$. The quotient $G/K$ will also be 
denoted by $\widetilde{X}$ in the sequel. Let $\gL$ and $\kL$ be the Lie
algebras 
of $G$ and $K$. We let $\theta$ be the standard Cartan involution of $\gL$. The
lift 
of $\theta$ to $G$ will be denoted by the same latter. Let
$\gL=\kL\oplus\pL$ be the corresponding Cartan decomposition. Then the 
Killing form $B$ of $\gL$ defines an inner product on $\pL$. We consider the
inner product 
$\left<\cdot,\cdot\right>$ on $\pL$ which is given by $\frac{1}{4} B$.  The
tangent 
space of $\widetilde X$ at  $1K$ can be identified with $\pL$ and therefore the
inner product
$\left<\cdot,\cdot\right>$ defines an invariant metric on $\widetilde X$. This
metric 
is the metric of constant curvature $-1$.

Now we let $\Gamma$ be a discrete, torsion free subgroup of $G$ with
$\vol(\Gamma\backslash G)<\infty$ and we let
\begin{align*}
X=\Gamma\backslash \widetilde{X}. 
\end{align*}
We equip $X$ with
the Riemannian metric
induced from $\widetilde{X}$.
Let $\mathfrak{P}$ be a fixed
set of representatives of $\Gamma$-nonequivalent proper cuspidal parabolic 
subgroups of
$G$. Then $\mathfrak{P}$ is finite. Throughout this paper we assume that for
every $P\in\mathfrak{P}$ with Langlands
decomposition $P=M_PA_PN_P$ one has
\begin{align}\label{asGamma}
\Gamma\cap P=\Gamma \cap N_P.
\end{align}
This condition is satisfied for example if $\Gamma$ is ``neat'', which means 
that the group generated by the eigenvalues of any $\gamma\in\Gamma$ contains 
no roots of unity $\ne1$. It also holds for many groups $\Gamma$ which are 
of arithmetic significance.
Let $\kappa(X):=\#\mathfrak{P}$. The geometric shape of $X$ can be described as
follows, see for example \cite{MP2}. There exists a $Y_{0}>0$ and for every
$Y\geq Y_{0}$
a compact manifold $X(Y)$ with smooth boundary such that $X$
admits a decomposition as 
\begin{align}\label{Zerlegung X}
X=X(Y)\cup \bigsqcup_{P\in\mathfrak{P}}F_{P,Y}
\end{align}
with
$X(Y)\cap F_{P,Y}=\partial X(Y)=\partial F_{P,Y}$ and $F_{P,Y}\cap
F_{P',Y}=\emptyset$ if $P\neq P'$. Here the $F_{P,Y}$ are the cusps of $X$. They
satisfy
$F_{P,Y}\cong
[Y,\infty)\times
T^2$, where $T^2$ denotes the flat 2-torus. Moreover the restriction of the
metric 
of $X$ to $F_{P,Y}$ is given as a warped product
$y^{-2}\frac{d^2}{dy^2}+y^{-2}g_{0}$, where 
$g_{0}$ denotes the suitably normalized standard-metric of $T^2$.

We let $P_0:=MAN$ be the standard parabolic subgroup of $G$. Then we have
$M=\SO_2(\R)$. By $\mL$, $\aL$ and
$\nL$ we denote the Lie algebras of $M$, $A$ and $N$. Then
$\hL:=\aL\oplus \mL$ 
is a Cartan subalgebra of $\gL$ and $\mL$ is a Cartan subalgebra of $\kL$. We
let $e_1\in \aL^*$ denote the restricted root which is implicit 
in the choice of $\nL$ and we fix $e_2\in i\mL^*$ such that positive roots 
$\Delta^+(\gL_\C,\hL_\C)$ can be defined by $\Delta^+(\gL_\C,\hL_\C):=\{e_1+e_2,
e_1-e_2\}$, see \cite[section 2]{MP1}. We let $H_1\in\aL$ be such that
$e_1(H_1)=1$.

By $\hat{M}$ and $\hat{K}$ we denote the equivalence classes of 
finite-dimensional irreducible representations of $M$ respectively $K$. For 
$\nu\in\hat{K}$, $\sigma\in\hat{M}$ we denote the multiplicity of $\sigma$ in
$\nu|_M$ 
by $\left[\nu:\sigma\right]$. Then every representation in $\hat{M}$ is
one-dimensional and the elements 
of $\hat{M}$ will be parametrized as $\sigma_j$,
$j\in\frac{1}{2}\mathbb{Z}$. Here $\sigma_j$ denotes the representation 
of $M$ with highest weight $je_2$. For $l\in\frac{1}{2}\mathbb{N}$ we
let $\nu_l$ be the representation of $K$ with highest weight $le_2$. Then 
$\hat{K}$ is parametrized by the elements $\nu_l$, $l\in\frac{1}{2}\mathbb{N}$.
Our
parametrization is different from the one used in \cite{Muller2} but consistend
with 
the notation of \cite{MP1}, \cite{MP2}. For
$k\in\frac{1}{2}\mathbb{Z}$ 
we define a representation $w_0\sigma_k$ of $M$ by $w_0\sigma _k:=\sigma_{-k}$,
see \cite[section 2]{MP1}. The representation rings
of $M$ and 
$K$ will be denoted by $R(M)$ respectively $R(K)$. Then the following lemma
holds. 
\begin{lem}\label{branching}
Let $\nu=\nu_l$, $l\in\mathbb{N}$. Then for $\sigma\in\hat{M}$ one has
$\left[\nu:\sigma\right]=1$ 
if $\sigma=\sigma_k$, $k\in\mathbb{Z}$, $|k|\leq l$ and
$\left[\nu:\sigma\right]=0$ otherwise. 
Let $\sigma=\sigma_k$, $k\in\mathbb{Z}-\{0\}$. For $\nu\in\hat{K}$,
$\nu=\nu_{|k|}$ let
$m_\nu(\sigma)=1$. For $\nu=\nu_{|k|-1}$ let $m_\nu(\sigma)=-1$. Finally, 
for $\nu\notin \{\nu_{|k|},\nu_{|k|-1}\}$ let $m_\nu(\sigma)=0$. 
Then in $R(M)$ one has
$\sigma+w_0\sigma=\sum_{\nu\in\hat{K}}m_\nu(\sigma)\nu|_{M}$. 
\end{lem}
\begin{proof}
This follows from \cite[equations 4.1, 4.2]{Muller2}, taking the different
parametrizations into account.
\end{proof}
 
For $m\in\frac{1}{2}\mathbb{N}$ we let $\rho(m)$ denote the $2m$-th symmetric 
power of the standard representation of $G=\Sl_2(\C)$ over
$V_{\rho(m)}:=\Sym^{2m}
\C^2$. Then in the notations 
of \cite{MP1}, \cite{MP2}, $\rho(m)$ corresponds to the representation with 
highest weight $\Lambda(\rho(m)):=me_1+me_2$. By \cite[equation 2.9]{MP1} we
have $\rho(m)\neq\rho(m)_\theta$ for 
each $m$, where $\rho_\theta:=\rho\circ\theta$ for a representation $\rho$ of
$G$. For $q=0,1,2$ let
$\mu_q:MA\to\Gl(\Lambda^q\nL_\C^*)$ 
be the $q$-th exterior power of the adjoint representation of $MA$ on
$\nL_\C^*$. For $\lambda\in\C$ and $a\in A$, $a=\exp(Y)$, $Y\in\aL$ we
let $\xi_\lambda(a):=e^{\lambda e_1(Y)}$. Then the restriction $\rho(m)|_{MA}$
of $\rho(m)$ to 
$MA$ has the following property.

\begin{lem}\label{Kost}
In the representation ring of $MA$ one has
\begin{align*}
\sum_{q=0}^2(-1)^q q \mu_q\otimes
\rho(m)|_{MA}=\sigma_{m}\otimes\xi_{m+1}-\sigma_{m+1}\otimes
\xi_{m}+\sigma_{-m}\otimes \xi_{-(m+1)}-\sigma_{-(m+1) }\otimes
\xi_{-m}.
\end{align*}
\end{lem}
\begin{proof}
This Lemma is a special case of \cite[Corollary 2.6]{MP1}. It can also 
be checked by a direct computation. 
\end{proof}
\begin{bmrk}\label{rmrkkost}
If for $k\in\{0,1\}$ the representations $\sigma_{\rho(m),k}\in\hat{M}$ and the
$\lambda_{\rho(m),k}\in\R$ are as 
in \cite[section 8]{MP2}, then $\sigma_{\rho(m),0}=\sigma_m$,
$\lambda_{\rho(m),0}=m+1$ and $\sigma_{\rho(m),1}=\sigma_{m+1}$,
$\lambda_{\rho(m),1}=m$.
\end{bmrk}

\section{Selberg and Ruelle zeta functions}\label{seczeta}
In this section we briefly recall the definition and some properties of the
Selberg and Ruelle zeta
functions. For further details we refer to \cite[section 3]{Pf}.\\
We let $\CC(\Gamma)_{\s}$ denote the
semisimple conjugacy classes of $\Gamma$. 
If $\gamma\in\Gamma$ is semisimple and nontrivial, there exists a unique
$\ell(\gamma)>0$ and a $m_\gamma\in M$, which is 
unique up to conjugation in $M$, such that $\gamma$ is conjugate to
$\exp{(\ell(\gamma)H_1)}m_\gamma$. The number $\ell(\gamma)$ is the length of
the 
closed geodesic associated to the conjugacy class $\left[\gamma\right]$.
Moreover 
the centralizer $Z(\gamma)$ of $\gamma$ in $\Gamma$ is an infinite cyclic group.
The conjugacy class $\left[\gamma\right]$ is called 
prime if $\gamma$ is a generator of $Z(\gamma)$ or equivalently if
the closed geodesic corresponding to $\left[\gamma\right]$ is a prime geodesic.
Now for $\sigma\in\hat{M}$ the Selberg zeta function $Z(s,\sigma)$ is
defined as
\begin{align*}
Z(s,\sigma)=\prod_{\substack{\left[\gamma\right]\in\CC(\Gamma)_{\s}-\left[
1\right]\\
\left[\gamma\right]\:\text{prime}}}\prod_{k=0}^\infty\det{
\left(\Id-\sigma(m_\gamma)\otimes
S^k\Ad(m_\gamma\exp(\ell(\gamma)H_1))|_{\bar{\mathfrak{n}}}e^{
-(s+n)\ell(\gamma)}\right)}.
\end{align*}
By \cite[section 3]{Pf} the infinite product converges for $\Real(s)>2$ and by
\cite[Theorem 1]{Pf} the function $Z(s,\sigma)$ admits 
a meromorphic continuation to $\C$. \\
Next for $\sigma\in\hat M$ we define the twisted Ruelle zeta function
$R(s,\sigma)$ by
\begin{align}\label{DefRuelle1}
R(s,\sigma):=\prod_{\substack{\left[\gamma\right]\in\CC(\Gamma)_{\s}-\left[
1\right]\\
\left[\gamma\right]\:\text{prime}}}\det{\left(\Id-\sigma(m_\gamma)e^{
-s\ell(\gamma)}
\right)}.
\end{align}
The infinite product in \eqref{DefRuelle1} converges absolutely for
$\Real(s)>2$, see \cite[section 3]{Pf}.
Furthermore, if $\rho$ is a finite-dimensional irreducible complex
representation of 
$G$, we define the associated Ruelle zeta function $R_\rho(s)$ by 
\begin{align*}
R_\rho(s):=\prod_{\substack{\left[\gamma\right]\in\CC(\Gamma)_{\s}-\left[
1\right]\\
\left[\gamma\right]\:\text{prime}}}\det{\left(\Id-\rho(\gamma)e^{-s\ell(\gamma)}
\right)}.
\end{align*}
This inifinite product converges absolutely for $\Real(s)$ sufficiently large,
see \cite[section 3]{Pf}.
By \cite[Corollary 1.2]{Pf} the functions $R_\sigma(s)$ and $R_\rho(s)$ have a
meromorphic
continuation to $\C$. We will also consider symmetric Selberg and Ruelle zeta
functions. 
For $\sigma$ the trivial representation of $M$ we let $S(s,\sigma):=Z(s,\sigma)$
and $R_{\sym}(s,\sigma):=R(s,\sigma)$. 
If $\sigma$ is non-trivial, we let $S(s,\sigma):=Z(s,\sigma)Z(s,w_0\sigma)$ and
$R_{\sym}(s,\sigma):=R(s,\sigma)R(s,w_0\sigma)$. 

\section{The regularized trace and the regularized determinant}\label{sectrdet}
In this section we define the regularized trace and the regularized 
analytic torsion. For further details we
refer 
the reader to section 4 and section 5 of \cite{MP2}.

Let us first introduce the differential operators we consider.
For a finite-dimensional unitary representation $\nu$ of $K$ over $V_\nu$ 
let $\widetilde{E}_\nu:=G\times_\nu V_\nu$
be the associated homogeneous vector bundle over $\widetilde X$. Let
$E_\nu:=\Gamma\backslash \widetilde{E}_\nu$ 
be the corresponding locally homogeneous vector bundle over $X$. We equip 
$\widetilde{E}_\nu$ with the $G$-invariant metric induced from
the metric on
$V_\nu$. This metric
pushes down to a metric
on $E_\nu$. The smooth sections of $\widetilde
E_\nu$ 
can be canonically identified with the space
\begin{align}\label{defsect}
C^{\infty}(G,\nu):=\{f:G\rightarrow V_{\nu}\colon f\in
C^\infty,\;
f(gk)=\nu(k^{-1})f(g),\,\,\forall g\in G, \,\forall k\in K\}.
\end{align}
We define the space $L^2(G,\nu)$ in the same way. Let $\widetilde A_\nu$ be the
differential operator on $\widetilde E_\nu$ which acts on $C^{\infty}(G,\nu)$ by
$-\Omega$. Then
by the arguments of \cite[section 4]{MP2} the operator 
$\widetilde A_\nu$ with domain the compactly supported functions in
$C^{\infty}(G,\nu)$ is essentially selfadjoint on $L^2(G,\nu)$  and bounded from
below. Its 
selfajoint closure will be denoted by $\widetilde A_\nu$ too. There exists a
smooth 
$\End(V_{\nu})$-valued function $H_t^\nu$ which belongs to all
Harish-Chandra-Schwarz spaces and which satisfies
${H}^{\nu}_{t}(k^{-1}gk')=\nu(k)^{-1}\circ {H}^{\nu}_{t}(g)\circ\nu(k')$ for all
$k,k'\in K$ an for all $g\in G$ such that $e^{-t\tilde{A}_\nu}$ acts
on $L^2(G,\nu)$ as a convolution
operator 
with kernel $H^\nu_t$, see \cite[equation 4.7]{MP2}. 
If $C^\infty(\Gamma\backslash G,\nu)$ are the $\Gamma$-invariant elements of
$C^\infty(G,\nu)$, 
then the smooth sections of $E_\nu$ can be identified with
$C^\infty(\Gamma\backslash G,\nu)$. Similarly, the square-integrable sections 
of $E_\nu$ can be identified with the $\Gamma$-invariant elements
$L^2(\Gamma\backslash G,\nu)$ 
of $L^2(G,\nu)$. 
Let $A_\nu$ be the differential operator on $E_\nu$ which acts as
$-\Omega$ on $C^\infty(\Gamma\backslash G,\nu)$. Then $A_\nu$ with domain 
the compactly supported elements in $C^\infty(\Gamma\backslash G,\nu)$ is again
bounded from below and
essentially selfadjoint on $L^2(\Gamma\backslash G,\nu)$ and its closure will be
denoted by the same symbol. Let 
$\lambda_{\nu,0}\leq \lambda_{\nu,1}\leq\dots$ be the sequence of eigenvalues of
$A_\nu$, counted 
with multiplicity. One can easily extend Theorem I.1 of \cite{Donnelly} and its
proof to the operators $A_\nu$ and thus there exists a constant
$C>0$ such that for each $\lambda>0$ one has 
\begin{align}\label{WLaw}
\#\{j\colon \lambda_{\nu,j}\leq \lambda\}\leq C(1+\lambda)^{\frac{3}{2}}. 
\end{align}
Now consider the heat-semigroup  $e^{-tA_\nu}$ of $A_\nu$ acting on
$L^2(\Gamma\backslash G,\nu)$. The operator
$e^{-tA_\nu}$ 
is an integral operator on $L^2(\Gamma\backslash G,\nu)$ with smooth kernel
$H^\nu(t;x,x')$ defined in \cite[equation 4.8]{MP2}.
Let $h^\nu(t;x,x'):=\Tr H^\nu(t;x,x')$. The
operator 
$e^{-tA_\nu}$ is not trace class and
$h^\nu(t;x,x)$ is 
not integrable over $X$. However, it follows from the
Mass-Selberg relations, that with respect to the decomposition \eqref{Zerlegung
X} the integral of $h^\nu(t;x,x)$ over
$X(Y)$ 
has an asymptotic expansion in $Y$ as $Y\to\infty$ and, following ideas of
Melrose, 
one can take the finite part in this expansion as a 
definition of the regularized trace $\Tr_{\reg}(e^{-tA_\nu})$ of $e^{-tA_\nu}$. 
Explicitly, one obtains 
\begin{align}\label{regtrace2}
\Tr_{\reg}\left(e^{-tA_\nu}\right)
&=\sum_{j}e^{-t\lambda_{\nu,j}}+
\sum_{\substack{\sigma\in\hat{M};\sigma=w_0\sigma\\
\left[\nu:\sigma\right]\neq
0}}e^{tc(\sigma)}\frac{\Tr(\widetilde{\boldsymbol{C}}(\sigma,\nu,0))}{4}
\nonumber\\
&-\frac{1}{4\pi}\sum_{\substack{\sigma\in\hat{M}\\
\left[\nu:\sigma\right]\neq
0}}\int_{\R}e^{-t\left(\lambda^2-c(\sigma)\right)}
\Tr\left(\widetilde{\boldsymbol{C}}(\sigma,\nu,-i\lambda)
\frac{d}{dz}\widetilde{\boldsymbol{C}}(\sigma,\nu,i\lambda)\right)\,d\lambda,
\end{align}
see \cite[equation 5.2, definition 5.1]{MP2}. Here the first sum on the right
hand side of
\eqref{regtrace2} converges absolutely 
by \eqref{WLaw} and all integrals converge absolutely by the arguments of
\cite[section 5]{MP2}. The
functions 
$\widetilde{\boldsymbol{C}}(\sigma,\nu,z)$ are meromorphic functions of $z$ with
 values in the endomorphisms of a finite-dimensional vector-space, which are
regular and invertible on $i\R$. They are
constructed 
out of the constant term-matrices, also called scattering matrices, associated
to the Eisenstein series, see
\cite[section 3, section 5]{MP2}. The constansts $c(\sigma)$ are 
defined by $c(\sigma_j):=j^2-1$, $j\in\frac{1}{2}\mathbb{Z}$.

The key fact which makes the regularized trace
accessible to computations 
is that the right hand side of \eqref{regtrace2} equals the spectral 
side of the Selberg trace formula applied to the function $h_t^\nu:=\Tr
H_t^\nu$. The spectral side of the trace formula consists of a sum of tempered 
distributions. We shall now define these distributions in 
the form in which they will be used for the subsequent computations.  For 
further details we refer to \cite[section 6]{MP2}. If $\sigma\in\hat{M}$,
$\lambda\in\C$, 
we let $\pi_{\sigma,\lambda}$ be the principle-series representation of $G$ as
in \cite[section 2.7]{MP2}. 
Then $\pi_{\sigma,\lambda}$ is unitary iff $\lambda$ is real. The global
character of $\pi_{\sigma,\lambda}$ 
will be denoted by $\Theta_{\sigma,\lambda}$. Let $\alpha$ be a $K$-finite
Schwarz
function. 
The identity and the hyperbolic term are defined by
\begin{align*}
I(\alpha):=\vol(X)\sum_{\sigma\in\hat{M}}\int_{\R}\Theta_{\sigma,\lambda}
(\alpha)P_\sigma(i\lambda)d\lambda;\quad H(\alpha):=\int_{\Gamma\backslash
G}\sum_{\gamma\in\Gamma_s-1}\alpha(x^{-1}\gamma x)dx.
\end{align*}
Here $P_\sigma$ is the Plancherel polynomial. Explicitly, for
$k\in\frac{1}{2}\mathbb{Z}$ one has
\begin{align}\label{Plnchrl}
P_{\sigma_{k}}(z)=\frac{1}{4\pi^2}(k^2-z^2),
\end{align}
see \cite{MP1}, \cite{Muller2}. Moreover, $\Gamma_{\s}$ are the semisimple
elements of 
$\Gamma$.
Next for each $\sigma$ in 
$\hat{M}$ we define a meromorphic function $\Omega(\sigma,\lambda)$ as in 
\cite[Theorem 6.2]{MP2} and we define the constant $C(\Gamma)$ as in \cite[page
22]{MP2}. Then the 
distributions $\mathcal{I}$ and $T$ are defined as 
\begin{align*}
\mathcal{I}(\alpha):=\frac{\kappa(X)}{4\pi}\sum_{\sigma\in\hat{M}}\int_{\mathbb{
R}}
\Theta_{\sigma, \lambda}(\alpha)\Omega(\check{\sigma},-\lambda)d\lambda;\quad
T(\alpha):=\frac{C(\Gamma)}{2\pi}\sum_{\sigma\in\hat{M}}\int_{\R}\Theta_{\sigma,
\lambda}
(\alpha)d\lambda
\end{align*}
Finally let $J_{P_0|\bar{P}_0}(\sigma,z)$ be the Knapp-Stein intertwining
operator defined as in \cite[equation 6.6]{MP2}. Then
$J_{P_0|\bar{P}_0}(\sigma,z)$ is a meromorphic function of $z\in\C$ 
which is regular and invertible on $\R-\{0\}$. Let $H_{\epsilon}$ be the
half-circle from $-\epsilon$ to $\epsilon$ in the lower half-plane, oriented
counter-clockwise. Let $D_{\epsilon}$ be the path which is the union of
$\left(-\infty,-\epsilon\right]$, $H_{\epsilon}$ and
$\left[\epsilon,\infty\right)$. Then the distribution $J$ is defined by
\begin{align}\label{DefJ}
J(\alpha):=-\sum_{\sigma\in\hat{M}}\frac{\kappa(X)}{4\pi
i}\int_{D_{\epsilon}}{\Tr\left(J_{\bar{P}_{0}|P_{0}}(\sigma,\zeta)^{-1}\frac{d}{
d\zeta}J_{\bar{P}_{0}|P_{0}}(\sigma,\zeta)\pi_{\sigma,\zeta}
(\alpha)\right)d\zeta}.
\end{align}
By the Selberg trace formula, one can express the regularized
trace as
\begin{align}\label{Trfrml}
\Tr_{\reg}(e^{-tA_\nu})=I(h^{\nu}_{t})+H(h^{\nu}_{t})+T(
h^{\nu}_{t})+\mathcal{I}(h^{\nu}_{t})+J(h^{\nu}_{t}),
\end{align}
see \cite[Theorem 6.1, Theorem 6.2]{MP2}.

Next we introduce the spectral
zeta function associated to $A_{\nu}+s$ for certain $s\in\C$. If
$\lambda_\nu\in\R$ is the smallest eigenvalue
of $A_{\nu}$, we define $b(\nu)\in\R$ by
\begin{align}\label{bnu}
b(\nu):=\max\left\{\{c(\sigma)\colon\sigma\in\hat{M}\colon\left[\nu:\sigma\right
]\neq
0\}\sqcup\{-\lambda_\nu\}\right\},
\end{align}
where the constants $c(\sigma)$ are as above. 

\begin{prop}\label{Ancontxi}
Let $s\in\C$
with
$\Real(s)>b(\nu)$. Then for
$\Real(z)>\frac{d}{2}$ the integral
\begin{align*}
\xi_{\nu}(s,z):=\int_{0}^{\infty}{t^{z-1}\Tr_{\reg}(e^{-t(A_{\nu}+s)})dt
}
\end{align*} 
converges and $\xi_{\nu}$ is holomorphic on
$\{(s,z)\in\C\times\C\colon\Real(s)>b(\nu)\colon\Real
(z)>\frac{d}{2}\}$ . 
Moreover, $\xi_{\nu}(s,z)$ has a continuation to a holomorphic function on
$\{(s,z)\in\C\times\C\colon\Real(s)>b(\nu)\colon z\neq -j, z\neq
3/2-j,j\in\N_0\}$. For every $s\in\C$ with
$\Real(s)>b(\nu)$
the function $z\mapsto\xi_{\nu}(s,z)$ is a meromorphic function on $\C$ with an
at most simple pole at $z=0$
and its residue at $z=0$ is independent of $s$. 
\end{prop}
\begin{proof}
By \eqref{regtrace2} there exists a constant $C$ such that 
$\left|\Tr_{\reg}e^{-t(A_\nu+s)}\right|\leq C e^{-t(\Real(s)-b(\nu))}$.
Thus the integral
$\int_{1}^{\infty}t^{z-1}\Tr_{\reg}e^{-t(A_\nu+s)}$
converges absolutely for all $\{(z,s)\in\C\times\C\colon\Real(s)>b(\nu)\}$
and is holomorphic there.
Expanding $e^{-ts}$ in a power series, it follows from \cite[Proposition
6.9]{MP2} that one has an asymptotic expansion
\begin{align}\label{Asym}
\Tr_{\reg}e^{-t(A_\nu+s)}\sim
\sum_{j=0}^\infty a_{j}(s)t^{j-\frac{3}{2}}+\sum_{j=0}^\infty b_{j}(s)t^{
j-\frac{1}{2}}\log{t}+\sum_{j=0}^\infty c_j(s)t^j
\end{align}
as $t\to +0$ which holds locally uniformly in $s$. 
Here the coefficients $a_j(s)$, $b_j(s)$ and $c_j(s)$ depend holomorphically on
$s$ and by \cite[Proposition
6.9]{MP2} and the fact that $d=3$ is odd it follows that
$c_0(s)$ is independent of $s$. Thus the Proposition follows from
standard methods which are described for
example in \cite{Gilkey}.
\end{proof}
Now we can define the regularized determinant proceeding as on a closed
manifold. By Proposition \ref{Ancontxi},
for $s\in\C$ with $\Real(s)>b(\nu)$ the function $\xi_{\nu}(s,z)/\Gamma(z)$
is regular at $z=0$. Thus
for $s\in\C$ with $\Real(s)>b(\nu)$ we 
define the 
determinant of $A_{\nu}+s$ by
\begin{align*}
{\det}(A_{\nu}+s):=\exp{\left(-\frac{\partial}{\partial
z}\bigr|_{z=0}\frac{
\xi_{\nu}(s,z)}{\Gamma(z)}\right)}.
\end{align*}
This definition generalizes the definition of the zeta-regularized
determinant of a positive elliptic differential operator on a closed manifold.
We
remark that one can easily show that $-b(\nu)$ equals 
the infimum of the spectrum of $A_\nu$. This fact puts the definition of
$b(\nu)$ into 
a natural context. However it will not be used here. 

We finally turn to the definition of the analytic torsion. For further details
we
refer to \cite[section 7]{MP2}. 
Let $\rho$ be a finite-dimensional irreducible complex representation of $G$
which 
is not invariant under $\theta$. Let $E'_{\rho}$ be the flat
vector bundle associated to the
restriction of $\rho$ to $\Gamma$. Then $E'_{\rho}$ is canonically isomorphic to
the locally homogeneous vector bundle $E_{\rho}$ associated to $\rho|_{K}$.
For 
$p=0,\dots,3$ we define
$\nu_{p}(\rho):=\Lambda^{p}\Ad^{*}\otimes\rho:\:K\rightarrow\Gl(\Lambda^{p}
\mathfrak{p}^{*}\otimes V_{\rho})$.
There is a canonical isomorphism
$\Lambda^{p}(E_{\rho})\cong\Gamma\backslash(G\times_{\nu_{p}(\rho)}(\Lambda^{p}
\mathfrak{p}^{*}\otimes V_{\rho}))$.
By \cite[Lemma 3.1]{Mats}, the bundle $E_\rho$ carries a canocial invariant
metric, called admissible
metric and 
if $\Delta_p(\rho)$ denotes 
the corresponding flat Hodge-Laplace operator acting on the $E_\rho$-valued 
$p$-forms, then by Kuga's formula $\Delta_p(\rho)$ acts on
$C^{\infty}(\Gamma\backslash G,\nu_{p}(\rho))$
as $-\Omega+\rho(\Omega)$, see \cite[equation 6.9]{Mats}. By \cite[Lemma 7.1
(2), Lemma
7.3]{MP2}, 
for $p=0,\dots,3$ one has $\rho(\Omega)-b_{\nu_p(\rho)}\geq \frac{1}{4}$, 
where the $b_{\nu_p(\rho)}$ are as in \eqref{bnu}. Thus the determinants 
$\det(\Delta_p(\rho)):=\det(A_{\nu_p(\rho)}+\rho(\Omega))$ are
defined. As in the closed case we now define the analytic torsion of $E_{\rho}$
by 
\begin{align*}
T_X(\rho):=\prod_{p=0}^{3}\det\Delta_{p}(\rho)^{(-1)^{p+1}\frac{p}{2}} .
\end{align*}
We define a $K$-finite Schwarz-function $k_t^{\rho}$ by 
\begin{align}\label{defk}
k_t^{\rho}:=e^{-t\rho(\Omega)}\sum_{p=0}^{3}(-1)^{p}ph_t^{\nu_p(\rho)}.
\end{align}
Then if we apply equation \eqref{Trfrml}, we obtain
\begin{align}\label{eqk}
\log{T_X(\rho)}=\frac{1}{2}
\frac{d}{dz}\biggl(\frac{1}{\Gamma(z)}\int_{0}
^\infty
t^{z-1}(I(k_t^{\rho})+H(k_t^{\rho})+T(k_t^{\rho})+\mathcal{I}(k_t^{
\rho})+J(k_t^{\rho}))dt\biggr)\biggr|_{z=0},
\end{align}
where the right hand side is defined near $z=0$ by analytic continuation of the
Mellin
transform.

\section{The determinant formula for the Symmetric Selberg zeta
function}\label{secdetfrml}
In this section we let $\sigma=\sigma_k$, $k\in\mathbb{N}$. We want to relate
the symmetric Selberg zeta function
$S(s,\sigma)$ to the graded
determinant of certain Laplace-type operators.
 
We consider the differential operator
$A(\sigma)$ which was introduced by Bunke
and Olbrich for the closed case \cite[section 1.1.3]{Bunke} and which had been
used in \cite[section 7]{Pf}.
Let
us briefly recall
its definition.
We let the $m_\nu(\sigma)$ be as
in 
Lemma \ref{branching}. Then one defines a vector bundle
$E(\sigma)$ over $X$ and a differential operator $A(\sigma)$ on $E(\sigma)$ by 
\begin{align}\label{Definition des Hilfsbundels}
E(\sigma):=\bigoplus_{\substack{\nu\in\hat{K}\\ m_{\nu}(\sigma)\neq 0}}
E_{\nu};\quad A(\sigma):=\bigoplus_{\substack{\nu\in\hat{K}\\
m_{\nu}(\sigma)\neq
0}}A_{\nu}+c(\sigma),
\end{align}
where $c(\sigma)$ is as in the preceding section.
We define a K-finite Schwarz function $h_t^\sigma$ by
\begin{align}
h_t^\sigma:=e^{-tc(\sigma)}\sum_{\nu\in\hat{K}}m_\nu(\sigma)h_t^\nu,
\end{align}
where the $h_t^\nu$ are as in the previous section. Then by
\cite[equation
7.4]{Pf} for $\sigma'\in\hat{M}$
one has
\begin{align}\label{FTh}
\Theta_{\sigma',\lambda}(h_t^\sigma)=e^{-t\lambda^2},\:\text{if
$\sigma'\in\{\sigma,w_0\sigma\}$};\quad
\Theta_{\sigma',\lambda}(h_t^\sigma)=0,\:\:\text{if
$\sigma'\notin\{\sigma,w_0\sigma\}$}
\end{align}

The bundle $E(\sigma)$ admits a grading 
$E(\sigma)=E^{+}(\sigma)\oplus E^{-}(\sigma)$
defined by the sign of $m_{\nu}(\sigma)$.
In this section we study the relative graded determinant of the operators
$A(\sigma)+s$.  To define this determinant, we start with the following Lemma.
\begin{lem}\label{Lemdet}
For $\nu\in\hat{K}$,
$m_\nu(\sigma)\neq 0$ one has $c(\sigma)\geq b(\nu)$. 
\end{lem}
\begin{proof}
By Lemma \ref{branching}, and the definition of the $c(\sigma)$, we have
$c(\sigma)\geq c(\sigma')$ for
every
$\sigma'\in\hat{M}$ with $m_\nu(\sigma)\left[\nu:\sigma'\right]\neq 0$,
$\nu\in\hat{K}$. Moreover, since $\sigma\neq w_0\sigma$, the twisted Dirac
operator $D(\sigma)$ can be
defined as in \cite[section 8, section 9]{Pf} and it follows from
\cite[Proposition 8.1]{Pf}
that
$A(\sigma)=D(\sigma)^2$. Thus the eigenvalues of $A(\sigma)$ are nonnegative and
the Lemma follows.
\end{proof}
Let $s\in\C$ with
$\Real(s)>0$. By Proposition \ref{Ancontxi} and Lemma \ref{Lemdet}, for
every $\nu\in\hat{K}$
with $m_\nu(\sigma)\neq 0$ the relative determinant
$\det{\left(A_\nu+c(\sigma)+s\right)}\in\C^*$ is defined.
Thus we can define the graded determinant
${\det}_{\gr}(A(\sigma)+s)\in\C^*$ of $A(\sigma)+s$
by
\begin{align*}
{\det}_{\gr}(A(\sigma)+s):=\prod_{\substack{\nu\in\hat{K}\\
m_\nu(\sigma)\neq
0}}\left({\det}{\left(A_\nu+c(\sigma)+s\right)}\right)^{m_\nu(\sigma)}
.
\end{align*}
We now study the function
$s\mapsto{\det}_{\gr}{\left(A(\sigma)+s^2\right)}$, $\Real(s)>0$,
$\Real(s^2)>0$. By \eqref{Trfrml} we have
\begin{align}\label{eqdet}
&\log{\det}_{\gr}(A(\sigma)+s^2)\nonumber\\
=&-\frac{d}{dz}\biggr|_{z=0}\biggl(\frac{1}{\Gamma(z)}\int_0^\infty
t^{z-1}e^{-ts^2}\left(I(h^{\sigma}_{t})+H(h^{\sigma}_{t})+T(
h^{\sigma}_{t})+\mathcal{I}(h^{\sigma}_{t})+J(h^{\sigma}_{t})\right)dt\biggr),
\end{align}
 where the right hand side is defined near $z=0$ by analytic
continuation of the Mellin
transform. 
We will compute the Mellin transform of each summand on the right hand side
separately. In the
sequel, we shall
write $\mathcal{LM}$ to
indicate that the Laplace-Mellin transform of a function is taken, allthoug we
take the Laplace-transform
in $s^2$ rather than in $s$. Firstly, the idenditiy contribution is easily
treated.
\begin{prop}\label{PropId}
Let $s\in\C$, $\Real(s)>0$, $\Real(s^2)>0$. For 
$\Real(z)>3/2$ the integral
\begin{align*}
\mathcal{LM}I(s,z,\sigma):=\int_{0}^{\infty}t^{z-1}e^{-ts^2}I(h_t^{\sigma}
)dt
\end{align*}
converges absolutely. Moreover, $\mathcal{LM}I(s,z,\sigma)$ has a meromorphic
continuation to $z\in\C$
and
is regular at $z=0$. Let
$\mathcal{LM}I(s,\sigma):=\mathcal{LM}I(s,z,\sigma)\bigr|_{z=0}$.
Then one has 
\begin{align*}
\mathcal{LM}I(s,\sigma)=-4\pi\vol
(X)\int_{0}^{s}P_{\sigma}(r)dr.
\end{align*}
\end{prop}
\begin{proof}
Since the $P_\sigma(z)$ are even polynomials in $z$ of degree 2, it follows
from \eqref{FTh} and a change of variables that
$I(h_t^\sigma)=a_0t^{-\frac{3}{2}}+a_1t^{-\frac{1}{2}}$,
where 
$a_0, a_1\in\C$.
Thus for $s\in\C$, $\Real(s)>0$, $\Real(s^2)>0$ and $\Real(z)>3/2$ the function
$\mathcal{LM}I(s,z,\sigma)$ is defined and it extends to a 
meromorphic function of $z\in\C$ which is regular at $z=0$. Moreover 
the assignment $s\to\mathcal{LM}I(s,0,\sigma)$ is holomorphic on
$\{s\in\C\colon\Real(s)>0,\:\Real(s^2)>0\}$. Applying
\cite[Lemma 2, Lemma 3]{Fried}, the Lemma is proved for $s\in (0,\infty)$ and
thus it also follows for
general $s$.  
\end{proof}
Next we treat the hyperbolic contribution. For our purposes, it suffices to
prove the following estimate.
\begin{prop}\label{PropHyp}
Let $s\in(\sqrt{2},\infty)$. Then for every 
$z\in\mathbb{C}$ the integral
\begin{align*}
\mathcal{LM}H(s,z,\sigma):=\int_0^\infty t^{z-1}e^{-ts^2}H(h_t^{\sigma})dt
\end{align*}
converges absolutely and $\mathcal{LM}H(s,z,\sigma)$ is an entire function of
$z$. Let $\mathcal{LM}H(s,\sigma) :=\mathcal{LM}H(s,z,\sigma)\bigr|_{z=0}$.
Then there exists a constant $C$ such that one has
$\left|\mathcal{LM}H(s,\sigma)\right|\leq Cs^{-2}$.
\end{prop}
\begin{proof}
For 
$\gamma\in \Gamma_{\s}$ we let $\gamma_0$ be a generator of $Z(\gamma)$ and we
let
\begin{align*}
f(t):=\sum_{\left[\gamma\right]\in\CC(\Gamma)_{\s}-\left[
1\right]}e^{-2\ell(\gamma)}\ell(\gamma_0)\frac{\Tr(\sigma)(m_{\gamma}
)+\Tr(w_0\sigma)(m_\gamma)}
{\det\left(\Id-\Ad(m_\gamma
a_\gamma)|_{\bar{\mathfrak{n}}}\right)}
\frac{e^{-\ell(\gamma)^{2}/4t}}{(4\pi
t)^{\frac{1}{2}}}.
\end{align*}
Then, since $\check{\sigma}=w_0\sigma$, by \eqref{FTh} and \cite[equation
5.4]{Pf} we have $H(h_t^\sigma)=f(t)$. Thus by \cite[Proposition 10.2]{MP2}, it
remains to prove the estimate in $s$. By \cite[equation 10.8]{MP2} there is
$C_1$ such that
\begin{align*}
\int_{1}^{\infty}t^{-1}e^{-ts^{2}}|f(t)|\;dt
\le C_1 
e^{-\frac{s^2}{4}}.
\end{align*}
Moreover, by \cite[equation 10.12]{MP2} there exists a constant
$c>0$ such that 
for $0<t\leq 1$ one can estimate $|f(t)|\leq e^{-\frac{c}{t}}$. 
Thus, by partial integration we obtain
\begin{equation*}
\int_{0}^1 t^{-1}|f(t)|e^{-ts^2}dt\leq
C_2\int_{0}^{1}e^{-ts^{2}}e^{-\frac{c}{2t}}dt
=C_2\left(-\frac{1}{s^{2}}e^{-s^{2}}e^{-\frac{c}{2}}
+\frac{c}{2s^{2}}\int_{0}^{1}{t^{-2}e^{-ts^{2}}e^{-\frac{c}{2t}}dt}\right)
\end{equation*}
for some constant $C_2$. It follows that there exists a constant
$C_3$ such that
\begin{align*}
\int_{0}^1 t^{-1}|f(t)|e^{-ts^2}dt\leq C_3s^{-2}.
\end{align*}
This proves the proposition. 

\end{proof}
The contribution of the distribution $T$ is as follows. 
\begin{prop}\label{PropT}
Let $s\in\C$, $\Real(s^2)>0$, $\Real(s)>0$. For $\Real(z)>3/2$ the integral
\begin{align*}
\mathcal{LM}T(s,z,\sigma):=\int_{0}^{\infty}t^{z-1}e^{-ts^2}T(h_t^{\sigma}
)dt
\end{align*}
conveges absolutely. Moreover, the function
$z\mapsto\mathcal{LM}T(s,z,\sigma)$ has a
meromorphic continuation
to $\C$ which is regular at $0$.
Let $\mathcal{LM}T(s,\sigma):=\mathcal{LM}T(s,z,\sigma)\bigr|_{z=0}$.
Then one has
$\mathcal{LM}T(s,\sigma)=-2C(\Gamma)s$.
\end{prop}
\begin{proof}
By \eqref{FTh} and the definition of $T$ one has
$\mathcal{LM}T(s,z,\sigma)=\frac{C(\Gamma)}{
\sqrt{\pi}}s^{
-2z+1}\Gamma\left(z-\frac{1}{2}\right)$ and
the proposition follows. 
\end{proof}

For the invariant distribution $\mathcal{I}$ associated to the weighted orbital 
integral we have the following proposition. 
\begin{prop}\label{PropIi}
Let the meromorphic function $\Omega(\sigma,\lambda)$, $\sigma\in\hat{M}$,
$\lambda\in\C$ be defined 
as in \cite[Theorem 6.2]{MP2}. Then for $k\in\mathbb{N}$ the function
$\Omega(\sigma_k,\lambda)$ is given as
\begin{align*}
\Omega(\sigma_k,\lambda)=-2\gamma-\psi(1+i\lambda)-\psi(1-i\lambda)
-\sum_{1\leq
l<|k|}\frac{2l}{\lambda^2+l^2}-
\frac{|k|}{\lambda^2+k^2},
\end{align*}
where $\gamma$ is the Euler-Mascheroni constat and $\psi$ is the Digamma
function. 
Moreover, for $s\in\C$, $\Real(s)>0$,
$\Real(s^2)>0$ and  $z\in\C$, $\Real(z)>3/2$ the integral
\begin{align*}
\mathcal{LM}\mathcal{I}(s,z,\sigma):=\int_{0}
^\infty t^{z-1}e^{-ts^2} 
\mathcal{I}(h^{\sigma}_t)dt
\end{align*}
converges absolutely. The function
$z\mapsto\mathcal{LM}\mathcal{I}(s,z,\sigma)$ has a
meromorphic continuation to
$z\in\C$ with an at most simple pole
at $z=0$. Let $\mathcal{LM}\mathcal{I}(s,\sigma):=\frac{\partial}{\partial
z}|_{z=0}\frac{\mathcal{LM}\mathcal{I}(s,
z,\sigma)}{\Gamma(z)}$. Then there exists a constant $C_0$ which is independent
of $X$ such that for every $k\in\mathbb{N}$ one has
\begin{align*}
\mathcal{LM}\mathcal{I}(s,
\sigma_k)=\kappa(X)C_0+2\kappa(X)\gamma
s+2\kappa(X)\log\Gamma(s+k)+\kappa(X)\log{(s+k)}.
\end{align*}
\end{prop}
\begin{proof}
The statement about $\Omega(\sigma_k,\lambda)$ follows from an elementary
computation 
using the identity $\psi(z+1)=\frac{1}{z}+\psi(z)$. Thus the Proposition follows
from 
\eqref{FTh} and
\cite[Lemma 10.5, Lemma 10.6]{MP2}. Here we remark that
the assumption $c\in(0,\infty)$ in 
these Lemmas can be weakened to $c\in\C$, $\Real(c)>0$. The proofs remain the
same.  
\end{proof}

We finally turn to the contribution of the non-invariant distribution $J$.
\begin{prop}\label{PropJ}
Let the distribution $J$ be as in \eqref{DefJ} and let $k\in\mathbb{N}$. The one
has
\begin{align}\label{EqJ}
J(h_t^{\sigma_k})=\frac{\kappa(X)}{2\pi }\left(1+2\sum_{1\leq
j<|k|}e^{-t(k^2-j^2)}+e^{-tk^2}\right)\int_{\R}\frac{1}{i\lambda+k
}e^{-t\lambda^2
}d\lambda.
\end{align}
Let $s\in\C$, $\Real(s)>0$,
$\Real(s^2)>0$. For $\Real(z)>3/2$ the integral
\begin{align*}
\mathcal{LM}J(s,z,\sigma_k):=\int_{0}
^\infty t^{z-1}e^{-ts^2} 
J(h^{\sigma_k}_t)dt
\end{align*}
converges absolutely. Moreover, the function
$z\mapsto\mathcal{LM}J(s,z,\sigma_k)$ has a meromorphic continuation to
$\C$ with a
simple pole at $0$.
For $\mathcal{LM}J(s,\sigma_k):=\frac{\partial}{\partial
z}|_{z=0}\frac{\mathcal{LM}J(s,
z,\sigma_k)}{\Gamma(z)}$ one has
\begin{align*}
\mathcal{LM}J(s,\sigma_k) =&-2\kappa(X)\sum_{1\leq
j<k}\log{(\sqrt{s^2+k^2-j^2}
+k)}\\ &-\kappa(X)\log{(\sqrt{s^2+k^2}
+k)}-\kappa(X)\log{(s+k)}.
\end{align*}
\end{prop}
\begin{proof}
For $j\in\mathbb{Z}$, $l\in\mathbb{N}^0$, $|j|\leq l$ we let
$c_{\nu_l}(\sigma_j,z)$ be the
Harish-Chandra $c$-function associated 
to the representation $\nu_l$ and $\sigma_j$ defined by \cite[equation
6.7]{MP2}. Then by \cite[Appendix 2]{Cohn} one has
\begin{align*}
c_{\nu_l}(\sigma_j,z):=\frac{\Gamma(iz-j)\Gamma(iz+j)}{
\Gamma(iz-l)\Gamma(iz+l+1)},
\end{align*}
see also \cite[equation 6.8]{MP2}. By \cite[equation 6.14]{MP2} and the
definition of $h_t^{\sigma_k}$ one
has \begin{align*}
J(h_t^{\sigma_k})=
-\frac{\kappa(X) e^{-tc(\sigma_k)}}{4\pi
i}\sum_{\nu\in\hat{K}}\sum_{\sigma'\in\hat{M}}m_{\nu}(\sigma_k)\left[
\nu:\sigma'\right
]\int_{D_{\epsilon}}{e^{-t(\zeta^{2}-c(\sigma'))}c_{\nu}(\sigma':\zeta)^{-1}
\frac{d}{d\zeta}c_{\nu}
(\sigma':\zeta)d\zeta}.
\end{align*}
Thus if one applies Lemma \ref{branching}, equation \eqref{EqJ} follows. 
If one applies \cite[Lemma 10.5]{MP2} to \eqref{EqJ}, the formula for
$\mathcal{LM}J(s,\sigma_k)$ follows. 
Here we remark again that the condition $c\in(0,\infty)$ in \cite[Lemma
10.5]{MP2} can be weakened to the condition $c\in\C$, $\Real(c)>0$ 
without changing the proof.

\end{proof}
By the preceding proposition, each summand on the right hand side of
\eqref{eqdet} can
be integrated individually and we have
\begin{align}\label{eqdet2}
&\log{\det}_{\gr}(A(\sigma)+s^2)\nonumber\\=&-\mathcal{LM}I(s,\sigma)-\mathcal{
LM}
H(s,
\sigma)-\mathcal{LM}T(s,\sigma)-\mathcal{LM}\mathcal{I}(s,\sigma)-\mathcal{LM}
J(s,\sigma).
\end{align}
To proof our determinant formula, we will also need the following Lemma.
\begin{lem}\label{LemTr}
Let
$\Tr_{\reg}(e^{-tA(\sigma)}):=\sum_{\nu\in\hat{K}}m_\nu(\sigma)e^{-tc(\sigma)}
\Tr_{\reg}(e^{ -tA_\nu})$. Then for $s,s_1\in\C$ with $\Real(s)>0$,
$\Real(s_1)>0$ one has
\begin{align*}
&\int_{0}^\infty
(e^{-ts}-e^{-ts_1})\Tr_{\reg}(e^{-tA(\sigma)})dt\\=&\frac{d}{d\zeta}\log{
\det}_{\gr
}(A(\sigma)+\zeta)\bigr|_{
\zeta=s}-\frac{d}{d\zeta}\log{
\det}_{\gr}(A(\sigma)+\zeta)\bigr|_{
\zeta=s_1}.
\end{align*}
\end{lem}
\begin{proof}
For $\zeta\in\C$, $\Real(\zeta)>0$ and $z\in\C$ let
$\xi_{A(\sigma)}(\zeta,z):=\sum_{
\nu\in\hat{K}}m_\nu(\sigma)\xi_{\nu}(\zeta+c(\sigma),z)$,
where the $\xi_{\nu}$ are as in the previous section. Then by definition one has
\begin{align*}
\log{\left({\det}_{\gr}(A(\sigma)+s)\right)}=-\frac{\partial}{\partial
z}\frac{\xi_{A(\sigma)}(s,z)}{\Gamma(z)}\bigr|_{z=0}.
\end{align*}
By \eqref{regtrace2} and the choice of $s$ and $s_1$,
$(e^{-ts}-e^{-ts_1})\Tr_{\reg}(e^{-tA(\sigma)})$ decays exponentially for
$t\to\infty$ and by \eqref{Asym} one has
$(e^{-ts}-e^{-ts_1})\Tr_{\reg}(e^{-tA(\sigma)})=O(t^{-1/2})$ as $t\to 0$.
Thus the 
integral in the lemma exists. 
By Proposition \ref{Ancontxi}, there  exists a constant $\alpha(\sigma)$
such that
for all $\zeta$ with $\Real(\zeta)>0$ one has
$\Res\bigr|_{z=0}\xi_{A(\sigma)}(\zeta,z)=\alpha(\sigma)$
and since $\frac{1}{\Gamma(z)}=z+\gamma z^2+O(z^3)$ as $z\to 0$, for all such
$\zeta$ one has $\log{\det}_{\gr}(A(\sigma)+\zeta)=-\lim_{z\rightarrow
0}(\xi_{A(\sigma)}(\zeta,z)-\frac{\alpha(\sigma)}{z}+\gamma
\alpha(\sigma))$.
One has $\xi_{A(\sigma)}(\zeta,z+1)=-\frac{\partial}{\partial
\zeta}\xi_{A(\sigma)}(\zeta,z)$ by the definition of $\xi_{A(\sigma)}$ and by
meromorphic continuation.
Thus one has
\begin{align*}
&\int_{0}^\infty
(e^{-ts}-e^{-ts_1})\Tr_{\reg}(e^{-tA(\sigma)})dt=\lim_{z\to
0}\int_{0}^\infty
t^z(e^{-ts}-e^{-ts_1})\Tr_{\reg}(e^{-tA(\sigma)})dt\\ 
=&-\lim_{z\to 0}\left(\frac{\partial}{\partial
\zeta}\xi_{A(\sigma)}(\zeta,z)\bigr|_{\zeta=s}-\frac{\partial}{\partial
\zeta}\xi_{A(\sigma)}(\zeta,z)\bigr|_{\zeta=s_1}\right)\\
=&-\lim_{z\to 0}\left(\frac{\partial}{\partial
\zeta}(\xi_{A(\sigma)}(\zeta,z)-\frac{\alpha(\sigma)}{z}+\gamma
\alpha(\sigma))\bigr|_{\zeta=s}-\frac{\partial}{\partial
\zeta}(\xi_{A(\sigma)}(\zeta,z)-\frac{\alpha(\sigma)}{z}+\gamma
\alpha(\sigma))\bigr|_{\zeta=s_1}\right)\\
=&\frac{d}{d\zeta}\log{\det}_{\gr}(A(\sigma)+\zeta)\bigr|_{
\zeta=s}-\frac{d}{
d\zeta}\log{\det}_{\gr}(A(\sigma)+\zeta)\bigr|_{\zeta=s_1}.
\end{align*}
Here limit and differentation in the third line can be interchanged since the
function
$(\zeta,z)\mapsto\xi_{A(\sigma)}(\zeta,z)-\frac{\alpha(\sigma)}{z}+\gamma
\alpha(\sigma)$ is holomorphic for $\zeta\in\C$, $\Real(\zeta)>0$ and
$z$ in 
a neighbourhood of zero.
\end{proof}

Now we can state the determinant formula for the symmetric Selberg zeta
function, which is the main 
result of this section.
\begin{prop}\label{Detdrd}
Let $s\in\C$, $\Real(s)>0$, $\Real(s^2)>0$. Let
$\kappa(X)$ be the number of cusps of $X$ and let
$c_\Gamma:=2\left(C(\Gamma)-\gamma
\kappa(X)\right)$, 
where $C(\Gamma)$ is as above. Then there exists a constant $C_0$ which is
independent of $X$ such that 
for every  $k\in\mathbb{N}$ one has
\begin{align*}
S(s,\sigma_k)=&e^{\kappa(X)C_0}\:{\det}_{\gr}{\left(A(\sigma)+s^2\right)}\exp{
\left(-4\pi\vol(X)\int_{0}^s P_{\sigma_k}(r)dr\right)}
\Gamma^{2\kappa(X)}(s+k)\\
&\cdot (s+k)^{\kappa(X)}\exp{\left(\mathcal{LM}J(s,\sigma_k)-sc_\Gamma\right)}.
\end{align*}
\end{prop}
\begin{proof}
We fix $s_1\in\R$, $s_1>2$ and let $s\in\R$, $s>2$. We let $\sigma:=\sigma_k$.
Then we can apply \cite[equation 7.7]{Pf} with $N=2$, $s_2:=s$ and obtain
\begin{align*}
\int_{0}^\infty
(e^{-ts^2}-e^{-ts_1^2})H(h_t^\sigma)dt=\frac{1}{2s}\frac{d}{ds}S(s,\sigma)-\frac
{1
}{2s_1}\frac{d}{ds}S(s_1,\sigma).
\end{align*}
Thus if we apply \eqref{Trfrml}, \eqref{FTh} and equation \eqref{Plnchrl},
Proposition
\ref{PropId}, Proposition
\ref{PropT},
Proposition \ref{PropIi} and Propostion \ref{PropJ}, we obtain a
constant $a$, depending on $s_1$, such that
\begin{align*}
&\int_{0}^\infty
(e^{-ts^2}-e^{-ts_1^2})\Tr_{\reg}(e^{-tA(\sigma)})dt\\=&\frac{1}{2s}\frac{d}{ds}
\log{S(s,\sigma)}+
\frac{2\pi\vol(X)P_\sigma(s)}{s}+\frac{C(\Gamma)}{s}-\frac{
\kappa(X)\gamma+\kappa(X)\psi(s+k)}{
s}\\
&+\frac{\kappa(X)}{
2\sqrt{s^2+k^2}(\sqrt{s^2+k^2}+k)}+\sum_{1\leq
j<k}\frac{\kappa(X)}{\sqrt{s^2+k^2-j^2}(\sqrt{s^2+k^2-j^2}+k)}+a\\
=&\frac{1}{2s}\frac{d}{ds}\biggl(
\log{S(s,\sigma)}-\mathcal{LM}I(s,\sigma)-\mathcal{LM}
T(s,\sigma)-\mathcal{LM}\mathcal{I}(s,\sigma)-\mathcal{LM}J(s,
\sigma)\biggr)
+a.
\end{align*}
If we multiply this equation by $2s$ and apply Lemma \ref{LemTr}, we obtain
\begin{align}\label{eqSlbrg}
\log{S(s,\sigma)}=&\log{{\det}_{\gr}(A(\sigma)+s^2)}+\mathcal{LM}I(s,
\sigma)+\mathcal{LM
}
\mathcal{I}(s,\sigma)+\mathcal{LM}T(s,\sigma)+\mathcal{LM}J(s,\sigma)\nonumber\\
&+as^2+b.
\end{align}
for some constant $b$. Thus by \eqref{eqdet2} we have
$\log{S(s,\sigma)}=-\mathcal{LM}H(s,\sigma)+as^2+b$.
Applying \cite[equations 3.3, 3.4, 3.5]{Pf}, it follows that $\log{S(s,\sigma)}$
decays
exponentially as $\Real(s)\to\infty$. Since $\mathcal{LM}H(s,\sigma)$ 
tends to zero for $s\in\R$, $s\to\infty$ by Propostion \ref{PropHyp}, the
constants $a$ and $b$ are zero. If we apply Proposition \ref{PropId},
Proposition \ref{PropT} and Proposition \ref{PropIi} and to the right hand side
of
\eqref{eqSlbrg}, the 
Proposition follows for $s\in (2,\infty)$. By \cite[Theorem 1.1]{Pf}, the
function
$S(s,\sigma)$ 
has a meromorphic continuation to $\C$ and since all functions on the right hand
side of the equation in the Proposition are holomorphic in 
$s\in\C$, $\Real(s)>0$, $\Real(s^2)>0$ by Proposition \ref{Ancontxi}, Lemma
\ref{Lemdet} and Proposition \ref{PropJ}, the Proposition 
follows.
\end{proof}

\section{The functional equations}\label{secfe}
Let $\sigma\in\hat{M}$. In this section we
prove a functional equation for the 
symmetric Selberg zeta function $S(s,\sigma)$.\\
For $\nu\in\hat{K}$ and $\sigma\in\hat{M}$  with $\left[\nu:\sigma\right]\neq 0$
we define 
the space $\boldsymbol{\mathcal{E}}(\nu:\sigma
)$ and the operator 
$\mathbf{C}(\nu:\sigma:\lambda):\boldsymbol{\mathcal{E}}(\nu:\sigma
) \to\boldsymbol{\mathcal{E}}(\nu:w_0\sigma
)$ as in \cite[section 4]{Pf}. 
Let us first symmetrize the scattering matrices. For $\sigma\in\hat{M}$,
$\sigma\neq w_0\sigma$ and
$\nu\in\hat{K}$ we let
$\boldsymbol{\overline{\mathcal{E}}}(\boldsymbol{\sigma},\nu):=\boldsymbol{
\mathcal{E}} (\boldsymbol{\sigma},\nu)\oplus
\boldsymbol{\mathcal{E}}(\boldsymbol{w_0\sigma},\nu)$ and for $s\in\C$ we let
\begin{align*}
\mathbf{\overline{C}}(\sigma:\nu:s):\
\:\boldsymbol{\overline{\mathcal{E}}}
(\boldsymbol{\sigma},\nu)\rightarrow
\boldsymbol{\overline{\mathcal{E}}}(\boldsymbol{\sigma},\nu);\quad
{\mathbf{\overline{C}}}(\sigma:\nu:s):=\begin{pmatrix}0&\mathbf{C}
(w_0\sigma:\nu:s)\\ \mathbf{C}(
\sigma :\nu:s)&0 \end{pmatrix}.
\end{align*}
By the 
arguments of \cite[section 4]{Pf}, the function
$\left(\det{\mathbf{\overline{C}}}(\sigma:\nu:s)\right)^{\frac{1}{
\dim(\nu)}}$ 
is canonically defined. For $\sigma\in\hat{M}$ we let $\nu_\sigma$ be as in
\cite[section 4]{Pf}. Then, if
$\sigma=\sigma_k$ we have $\nu_\sigma=\nu_{\left|k\right|}$. To save notation,
for $k\in\frac{1}{2}\mathbb{N}$ we
shall write
\begin{align*}
\left(\det{\mathbf{\overline{C}}}(\sigma_k:\nu_k:s)\right)^{\frac{
1}{
\dim(\nu)}}
=:\mathbf{C}(k:s).
\end{align*}
Then $\mathbf{C}(k:s)$ is a meromorphic function of $s$ which has no zeroes and
poles for $s\in i\R$.  By \cite[equation 4.2]{Pf} it satisfies
$\mathbf{C}(k:s)\mathbf{C}(k:-s)=1$.\\
We can now state a functional equation for the symmetric Selberg
zeta function.
\begin{prop}\label{FGSZF}
Let $k\in\mathbb{N}$ and let $c_\Gamma$ be as in Proposition
\ref{Detdrd}. 
Then the symmetric Selberg zeta
function $S(s,\sigma_k)$ satisfies 
the functional equation
\begin{align*}
S(-s,\sigma_k)=&S(s,\sigma_k)\exp{\left(8\pi\vol(X)\int_{0}^s
P_{\sigma_k}(r)dr+2c_\Gamma
s\right)}\frac{\left(\Gamma(-s+k)\right)^{2\kappa(X)}}{\left(\Gamma(s+k)\right)^
{2\kappa(X)}}
\frac{\mathbf{C}(k:s)}{\mathbf{C}
(k:0)}.
\end{align*}
  
\end{prop}
\begin{proof}
Let
\begin{align*}
\Xi(s,\sigma_k)
:=\exp\left(4\pi{\rm{vol}}(X)\int_{0}^{s}{P_{\sigma_k}
(r)dr}
+sc_{\Gamma}\right)
\left(\Gamma\left(s+k\right)\right)^{-2\kappa(X)}\cdot
S(s,\sigma_k).
\end{align*}
We note that the Polynomial $Q(\sigma_k,\lambda)$ and the constants
$c_{j,l}(\sigma)$ 
occuring in \cite[Proposition 5.4]{Pf} are zero in the 3-dimensional case. Thus
if we combine \cite[Proposition 7.2]{Pf}, \cite[equation 4.2]{Pf} \cite[equation
4.10]{Pf} and \cite[Remark 4.3]{Pf}, we obtain
\begin{align*}
\frac{\Xi'(s,\sigma_k)}{\Xi(s,\sigma_k)}+\frac{\Xi'(-s,\sigma_k)}{\Xi(-s,
\sigma_k)}
=-\frac{d}{ds}\log\mathbf{C}(k:s).
\end{align*}
Hence the logarithmic derivative of
$\frac{\Xi(s,\sigma)}{\Xi(-s,\sigma)}\mathbf{C}(k:s)$
is zero and so this function is constant. Now the order of the singularity 
of the function $\Xi(s,\sigma_k)$ at $0$ is the same as the order of the
singularity of $S(s,\sigma)$ 
at $0$. This order is even by \cite[Theorem 9.2]{Pf}. Since
$P_{\sigma_k}(r)$ is 
an even polynomial, the proposition follows.
\end{proof}

The previous proposition implies the following functional equation for the 
symmetric Ruelle zeta function. 

\begin{prop}\label{Fgldrd}
Let $k\in\mathbb{N}$. Then the symmetric Ruelle zeta
function
$R_{\sym}(s,\sigma_k)$ satisfies
the functional equation
\begin{align*}
R_{\sym}(-s,\sigma_k)=&R_{\sym}(s,\sigma_k)\exp{\left(-\frac{8}{\pi}
\vol(X)s\right)}\frac{\mathbf{C}
(k:s-1)\mathbf{C}(k:s+1)}{\mathbf{C}(k+1:s)\mathbf{C}(k-1:s)}\\ 
&\cdot\frac{\mathbf{C}(k+1:0)\mathbf{C}(k-1:0)}{\mathbf{C}
(k:0)^2}.
\end{align*}

\end{prop}
\begin{proof}
The same argument as in \cite[Lemma 3.1]{Muller2} gives 
\begin{align*}
R_{\sym}(s,\sigma_k)=\frac{S(s+1,\sigma_k)S(s-1,\sigma_k)}{S(s,
\sigma_{k+1})S(s,\sigma_{k-1})}.
\end{align*}
Moreover, using \eqref{Plnchrl} we compute
\begin{align*}
\int_{0}^{s+1}P_{\sigma_k}(r)dr+\int_{0}^{s-1}P_{\sigma_k}(r)dr-\int_{0}^{s}P_
{\sigma_{k+1}}(r)dr-\int_{0}^{s}P_{\sigma_{k-1}}(r)dr=
-\frac{s}{\pi^2}.
\end{align*}
Thus the proposition follows from Proposition
\ref{FGSZF}.
\end{proof}

To prove Theorem \ref{Theorem2}, we will also need the following
proposition.

\begin{prop}\label{PropordR}
Let $m\in\mathbb{N}$, $m\geq 3$. Then 
\begin{align*}
R_{\rho(m)}(s)R_{\rho(m)_{\theta}}(s)=&R_{\rho(2)}(s)R_{\rho(2)_\theta}(s)
\frac{\mathbf{C}(m:m+1-s)}{
\mathbf{C}(m+1:m-s)}\frac{\mathbf{C}(3:2-s)}{\mathbf{C}(2:3-s)}\frac{\mathbf{C}
(m+1:0)}{\mathbf{C}(m:0)}\frac{\mathbf{C}
(2:0)}{\mathbf{C}(3:0)}\\ &\cdot\prod_
{
k=3}^mR_{\sym}(k-s,\sigma_k)R_{\sym}(k+s,\sigma_k)\exp{\left(-\frac{8}{\pi}
\vol{(X)}(k-s)\right)}.
\end{align*}
\end{prop}
\begin{proof}
Let $m\in\mathbb{N}$, $m\geq 3$. Applying \cite[equation (3.14)]{Muller2},
we
can symmetrize \cite[equation (8.2)]{Muller2} and
obtain
\begin{align*}
R_{\rho(m)}(s)R_{\rho(m)_\theta}(s)=R_{\rho(2)}(s)R_{{\rho(2)}_\theta}(s)\prod_{
k=3}^{m} R_{\sym}(s+k,\sigma_k)R_{\sym}(s-k,\sigma_k).
\end{align*}
Thus together with proposition \ref{Fgldrd}, the proposition follows.
\end{proof}

\section{Proof of the main results}\label{secpr}
In this section we prove our main results. Let $m\in\mathbb{N}$.
Arguing as in \cite[Proposition 3.5]{Muller2} it follows that
\begin{align}\label{GlRSdrd}
R_{\rho(m)}(s)R_{\rho(m)_{\theta}}(s)=\frac{S(s+m+1,\sigma_m)S(s-m-1,\sigma_m)}{
S(s+m,\sigma_{m+1})S(s-m,\sigma_{m+1})}.
\end{align}
Now we express the analytic torsion $T_X(\rho(m))$ by the graded determinants 
associated to the operators $A(\sigma)$ which were introduced section
\ref{secdetfrml}. Namely, we have the following Proposition which is a
generalization 
of \cite[equation 7.28]{Muller2} to the noncompact case. 
\begin{prop}\label{Torsdreid}
Let $m\in\mathbb{N}$. Then one has
\begin{align*}
T_X(\rho(m))^2=\frac{{\det}_{\gr}{
\left(A(\sigma_m)+(m+1)^2\right)}}{{\det}_{\gr}{
\left(A(\sigma_{m+1})+m^2\right)}}.
\end{align*}
\end{prop}
\begin{proof}
Let $k_t^{\rho(m)}$ be as in \eqref{defk}. Then by Lemma \ref{Kost} and Remark
\ref{rmrkkost},
as a special case of \cite[Proposition 8.2]{MP2}
one has
$k_t^{\rho(m)}=e^{-t(m+1)^2}h_t^{\sigma_m}-e^{-tm^2}h_t^{\sigma_{m+1}}$. 
Applying \eqref{eqk} and \eqref{eqdet}, the proposition follows. 
\end{proof}
In order to relate the behaviour of $R_\rho(m) R_{\rho(m)_{\theta}}$ at $0$ 
to the analytic torsion $T_X(\rho(m))$, we want to apply the
determinant formula for 
the symmetric Selberg zeta function from Proposition \ref{Detdrd} to the right
hand side of \eqref{GlRSdrd} and 
combine it with Proposition \ref{Torsdreid}.

However, in contrast to a closed hyperbolic
manifold, this 
is not possible directly since the determinant formula for the symmetric Selberg
zeta
function is
valid only for $s\in\C$ with $\Real(s)>0$, $\Real(s^2)>0$. Thus we first have to
apply the
functional 
equation from Proposition \ref{FGSZF}.
We obtain the following
proposition.
\begin{prop}\label{PropRSneu} For $m\in\mathbb{N}$ one has
\begin{align*}
R_{\rho(m)}(s)R_{\rho(m)_\theta}(s)=&e^{2c_\Gamma}\frac{S(-s+m+1,\sigma_{m}
)S(s+m+1,\sigma_{m
})
\mathbf{C}(m+1:0)}
{
S(-s+m,\sigma_{m+1})S(s+m,\sigma_{m+1})\mathbf{C}(m:0)}\\
&\cdot\frac{\mathbf{C}(m:m+1-s)(\Gamma(s-1))^{
2\kappa(X)}\exp{\left(8\pi
\vol(X)\int_0^{-s+m+1}P_{\sigma_{m}}(r)dr\right)}}{\mathbf{C}
(m+1:m-s)(\Gamma(s+1))^{
2\kappa(X)
}\exp{\left(8\pi
\vol(X)\int_0^{-s+m}P_{\sigma_{m+1}}(r)dr\right)}} .
\end{align*}
\end{prop}
\begin{proof}
By proposition \ref{FGSZF} we have
\begin{align*}
\frac{S(s-m-1,\sigma_{m})}{S(s-m,\sigma_{m+1})}=&e^{2c_\Gamma}\frac{S(-s+m+1,
\sigma_{m})\exp{\left(8\pi
\vol(X)\int_0^{-s+m+1}P_{\sigma_{m}}(r)dr\right)}}{S(-s+m,\sigma_{m+1})\exp{
\left(8\pi
\vol(X)\int_0^{-s+m}P_{\sigma_{m+1}}(r)dr\right)}}\\
&\cdot\frac{\mathbf{C}(m+1:0)\mathbf{C}(m:m+1-s)(\Gamma(s-1))^{
2\kappa(X)}}{\mathbf{C}(m:0)\mathbf{C}(m+1:m-s)(\Gamma(s+1))^{2\kappa(X)}}.
\end{align*}
Applying \eqref{GlRSdrd}, the proposition follows. 
\end{proof}

Now we can prove Proposition \ref{Prop1}. We shall state the proposition 
also for $m+\frac{1}{2}$, $m\in\mathbb{N}$. 
The proof remains the same if one makes the appropriate modifications 
in section \ref{secdetfrml} and in Proposition \ref{PropordR},
Proposition \ref{Torsdreid} and Proposition \ref{PropRSneu}.  For
$m\in\mathbb{N}$ 
we define 
\begin{align}\label{Defc}
c(m+1/2):=\frac{\prod_{j=0}^{m-1}\sqrt{(m+3/2)^2+(m+1/2)^2-(j+1/2)^2
}+m+1/2
}{\prod_{j=0}^{m
}\sqrt{(m+3/2)^2+(m+1/2)^2-(j+1/2)^2}+m+3/2}.
\end{align}
Then we have the following proposition.
\begin{prop}\label{RundTors}
For $m\in\mathbb{N}$ we define the constant $c(m)$s and $c(m+\frac{1}{2})$ as 
in Theorem \ref{Theorem1} resp. equation \eqref{Defc}.
Then one has
\begin{align*}
&{T_X(\rho(m))}^4
\\ =&c(m)^{4\kappa(X)}\frac{\mathbf{C}(m:0)}{\mathbf{C}(m+1:0)}\lim_{s\to
0}\left(R_{\rho(m)}(s)R_{\rho(m)_{\theta}}(s)\frac{\mathbf{C}(m+1:m-s)}{\mathbf{
C}(m:m+1-s)}(\Gamma(s-1))^{
-2\kappa(X)}\right)
\end{align*}
and
\begin{align*}
{T_X(\rho(m+1/2))}^4=&c(m+1/2)^{4\kappa(X)}\frac{\mathbf{C}(m+1/2:0)}{\mathbf{C}
(m+3/2:0)}
\cdot
\lim_{s\to
0}\biggl(R_{\rho(m+1/2)}(s)R_{\rho(m+1/2)_{\theta}}(s)\\ &\frac{\mathbf{C}
(m+3/2:m+1/2-s)}{\mathbf{C}(m+1/2:m+3/2-s)}(\Gamma(s-1))^{-
2\kappa(X)}\biggr).
\end{align*}
\end{prop}
\begin{proof}
Let $m\in\mathbb{N}$. To save notation, let us first introduce two
auxiliary functions.
Let 
\begin{align*}
P_{\rho(m)}(s):=&\exp\biggl(-4\pi\vol(X)\int_0^{s+m+1}P_{\sigma_m}
(r)dr+4\pi\vol(X)\int_0^{-s+m+1}P_{\sigma_m}(r)dr\\
&-4\pi\vol(X)\int_0^{-s+m}P_{
\sigma_{m+1}}(r)dr+4\pi\vol(X)\int_0^{s+m}P_{\sigma_{m+1}}(r)dr\biggr).
\end{align*}
Moreover, let
\begin{align*}
J_{\rho(m)}(s):=\exp\biggl(&\mathcal{LM}J(-s+m,\sigma_{m+1})+\mathcal{LM}J(s+m,
\sigma_{m+1})\\ &-\mathcal{LM}J(-s+m+1,\sigma_{m})-\mathcal{LM}J(s+m+1,
\sigma_{m})\biggr).
\end{align*}
Then by Proposition \ref{PropRSneu}, Proposition \ref{Detdrd} and Proposition
\ref{Torsdreid} one has
\begin{align*}
&\frac{\mathbf{C}(m:0)}{\mathbf{C}(m+1:0)}\lim_{s\to
0}\left(R_{\rho(m)}(s)R_{\rho(m)_{\theta}}(s)\frac{\mathbf{C}
(m+1:m-s)(\Gamma(s+1))^{
2\kappa(X)
}}{\mathbf{C}(m:m+1-s)(\Gamma(s-1))^{
2\kappa(X)}}J_{\rho(m)}(s)\right)\\ 
=&\lim_{s\to 0}\left(e^{2c_\Gamma}\frac{S(s+m+1,\sigma_{m})S(-s+m+1,\sigma_{m
})
}
{
S(s+m,\sigma_{m+1})S(-s+m,\sigma_{m+1})}J_{\rho(m)}(s)\right)\\
&\cdot\lim_{s\to 0}\exp{\left(8\pi
\vol(X)\int_0^{-s+m+1}P_{\sigma_{m}}(r)dr-8\pi\vol(X)\int_{0}^{-s+m}P_{\sigma_{
m+1}}(r)dr\right)}\\
=&\lim_{s\to
0}\frac{\det_{\gr}\left(A(\sigma_{m})+(s+m+1)^2\right)\det_{\gr}
\left(A(\sigma_{m})+(-s+m+1)^2\right)}{\det_{\gr}
\left(A(\sigma_{m+1})+(s+m)^2\right)\det_{\gr}
\left(A(\sigma_{m+1})+(-s+m)^2\right)}P_{\rho(m)}(s)\\
=&\frac{{\det}_{\gr}^2\left(A(\sigma_{m})+(m+1)^2\right)}{{\det}_{\gr}^2
\left(A(\sigma_{m+1})+m^2\right)}\\
=&T_X(\rho(m))^4.
\end{align*}
Here we used that the function $P_{\rho(m)}(s)$ is an entire function of
$s$ satisfying $P_{\rho(m)}(0)=1$. Now by Proposition
\ref{PropJ}  the function $J_{\rho(m)}(s)$ is entire for $s$ in a
neighbourhood of
zero and one has $J_{\rho(m)}(0)=c(m)^{4\kappa(X)}$. This
proves
the proposition for $m\in\mathbb{N}$.  For $m+1/2$ one can argue in the same
way.
\end{proof}

Let us finally turn to the proof of Theorem \ref{Theorem2}. We recall that 
the infinite products in \eqref{DefRuelle1} defining the Ruelle zeta functions
$R(s,\sigma)$ converge 
absolutely for $\Real(s)>2$. Let
$m\in\mathbb{N}$, $m\geq 3$.
By Proposition \ref{RundTors} and Proposition \ref{PropordR} we have
\begin{align*}
&T_X(\rho(m))^4\\
=&c(m)^{4\kappa(X)}\frac{\mathbf{C}(m:0)}{\mathbf{C}(m+1:0)}\lim_{
s\to
0}\left(R_{\rho(m)}(s)R_{\rho(m)_{\theta}}(s)\frac{\mathbf{C}(m+1:m-s)}{\mathbf{
C}(m:m+1-s)}(\Gamma(s-1))^{-2\kappa(X)
}\right)\\
=&c(m)^{4\kappa(X)}\frac{\mathbf{C}(2:0)}{\mathbf{C}(3:0)}\lim_{s\to
0}\left(R_{\rho(2)}(s)R_{\rho(2)_\theta}(s)
\frac{\mathbf{C}(3:2-s)}{\mathbf{C}(2:3-s)}(\Gamma(s-1))^{-
2\kappa(X)}\right)
\\
&\cdot\prod_{k=3}^{m}\exp{\left(-\frac{8}{\pi}\vol(X)k\right)}R_{\sym}(k,
\sigma_k)^2\\
=&\frac{c(m)^{4\kappa(X)}}{c(2)^{4\kappa(X)}}T_X(\rho(2))^4\exp{\left(-\frac{4}{
\pi}\vol(X)
(m(m+1)-6)\right)}\prod_{k=3}^{m}R_{\sym}(k,\sigma_k)^2.
\end{align*}
Now one has $\overline{\sigma}=w_0\sigma$ and so by the definition of the
Ruelle zeta function and by meromorphic
continuation one gets $\overline{R(\bar{s},w_0\sigma)}=R(s,\sigma)$.
Thus one has 
$R_{\sym}(k,\sigma_k)=\left|R(k,\sigma_k)\right|^2$. This proves the first
equation in Theorem \ref{Theorem2}. 
Modifying Proposition
6.3, the second
equation in this theorem is obtained in the same way.

\end{document}